%% file: main.tex
\documentclass[11pt,leqno]{amsart}
\usepackage[a4paper,top=2.5 cm,bottom=2 cm,left=2.5 cm,right=2 cm]{geometry}
\usepackage{graphicx}
\usepackage{upref,ulem}
\usepackage{textcomp} 
\usepackage{tikz} 
\usetikzlibrary{arrows,snakes,backgrounds,calc}
\usepackage{amsmath ,amssymb}
\usepackage{color,bbm,epsfig,float,fancyhdr,soul}
\usepackage[mathscr]{eucal}
\usepackage{caption}
\usepackage[makeroom]{cancel}
\DeclareMathOperator{\sgn}{sgn}

\newtheorem{definition}{Definition}[section]
\newtheorem{theorem}{Theorem}[section]
\newtheorem{lemma}{Lemma}[section]

\newtheorem{remark}{Remark}[section]
\newtheorem{proposition}{Proposition}[section]
\newtheorem*{maintheorem*}{Main Theorem}
\newtheorem{Alg}{Algorithm}[section]
\allowdisplaybreaks
\numberwithin{equation}{section}

\renewcommand{\i}{\ifmmode\mathit{\mathchar"7010 }\else\char"10 \fi}
\renewcommand{\j}{\ifmmode\mathit{\mathchar"7011 }\else\char"11 \fi}
\newcommand{\R}{\mathbb{R}}
\newcommand{\N}{\mathbb{N}}
\newcommand{\Z}{\mathbb{Z}}

\newcommand{\Cc}[1]{\mathbf{C_c^{#1}}}

\newcommand{\dd}{\mathrm{d}}

\newcommand{\BV}{\mathbf{BV}}
\renewcommand{\L}[1]{\mathbf{L^#1}}

\newcommand{\C}[1]{\mathbf{C^{#1}}}
{%

\begin{enumerate}}%
{\end{enumerate}}

{%

\begin{enumerate}}%
{\end{enumerate}}

\title{Nonlocal reaction traffic flow model with on-off ramps }
\begin{document}
\author{F. A. Chiarello} 
\address[Felisia Angela Chiarello]{\newline Department of Mathematical Sciences “G. L. Lagrange”, Politecnico di Torino, Corso Duca degli Abruzzi 24, 10129 Torino, Italy.}
\email[]{felisia.chiarello@polito.it}

\author{H. D. Contreras}
\address[Harold Deivi Contreras]{\newline
GIMNAP-Departamento de Matem\'aticas, Universidad del B\'io-B\'io, Concepci\'on, Chile,\newline CI${}^2$MA-Universidad de Concepci\'on, Casilla 160-C, Concepci\'on, Chile.}
\email[]{harold.contreras1801@alumnos.ubiobio.cl}

\author{L. M. Villada}
\address[Luis Miguel Villada]{\newline
GIMNAP-Departamento de Matem\'aticas, Universidad del B\'io-B\'io, Concepci\'on, Chile,\newline CI${}^2$MA-Universidad de Concepci\'on, Casilla 160-C, Concepci\'on, Chile.}
\email[]{lvillada@ubiobio.cl}

\date{\today}
\maketitle

\begin{abstract}
    We present a non-local version of a scalar balance law modeling traffic flow with on-ramps and off-ramps. The source term is used to describe the traffic flow over the on-ramp and off-ramps. We approximate the problem using an upwind-type numerical scheme and we provide $\L{\infty}$ and $\BV$ estimates for the sequence of approximate solutions. Together with a discrete entropy inequality, we also show the well-posedness of the considered class of scalar balance laws. Some numerical simulations illustrate the behaviour of solutions in sample cases.
\end{abstract}

\section{Introduction}
\subsection{Scope} Models of conservation laws with nonlocal flux are used to describe traffic flow dynamics in which drivers adapt their velocity with respect to what happens to the cars in front of them { \cite{Blandin2016WellposednessOA, Ch_G_global, GoatinScialanga, FKG2018, sopasakis2006stochastic}}. In this type of models, the flux function depends on a downstream convolution term between the density or the velocity of vehicles and a kernel function with support on the negative axis. However, the above models cannot be used to study the traffic flow on the highway with ramps since they did not include their presence. Indeed, ramps are an important element of traffic systems and develops some complex traffic phenomena, see \cite{han2020hierarchical,jacquet2005optimal,liptak2021traffic,sun2015study,tie2010effects,tie2009new,wang2017congested}. 

In this work, we propose a new nonlocal traffic model which includes the effects of on- and off- ramps. We start by considering a local reaction traffic model proposed in \cite{liptak2021traffic},
\begin{equation}\label{RTM}
    \rho_t+(\rho v(\rho))_x=S_{\mathrm{on}}-S_{\mathrm{off}},
\end{equation}
where the non-negative functions  $S_{\mathrm{on}}$ and $S_{\mathrm{off}}$ are the source and sink term, respectively, defined by
\begin{eqnarray}
\label{son_local}S_{\mathrm{on}}(t, x, \rho) &=&\mathbf{1}_{\mathrm{on}}(x)q_{\mathrm{on}}(t) (\rho_{\max}-\rho),\\
\label{soff_local}S_{\mathrm{off}}(t, x, \rho) &=&\mathbf{1}_{\mathrm{off}}(x)q_{\mathrm{off}}(t) \rho,
\end{eqnarray}

with $q_{\mathrm{on}}\in\R^{+}$, and $q_{\mathrm{off}}\in\R^{+}$ the rate of the on- and off-ramp respectively. The spatial position of the on- and off- ramp is described by indicator functions $\mathbf{1}_{\mathrm{on}}(x)$, and $\mathbf{1}_{\mathrm{off}}(x)$ defined as
\begin{equation*}
\mathbf{1}_{\mathrm{on}}(x)=
    \begin{cases}
    \frac{1}{L} \quad &\underline{x}_{\mathrm{on}}\leq x\leq \overline{x}_{\mathrm{on}},\\
    0 \quad&\text{otherwise},
    \end{cases}
    \qquad 
    \mathbf{1}_{\mathrm{off}}(x)=
    \begin{cases}
    \frac{1}{L} \quad &\underline{x}_{\mathrm{off}}\leq x\leq \overline{x}_{\mathrm{off}},\\
    0 \quad&\text{otherwise}.
    \end{cases}
\end{equation*}

In order to obtain a non-local version of the model \eqref{RTM}, we first rewrite the flux function $f(\rho)=\rho v(\rho)$ in its non-local version, see \cite{Amorim2013ANA, Blandin2016WellposednessOA, GoatinScialanga},
$$f(\rho)=\rho v(\rho*\omega_{\eta}), \quad \text{with}\quad (\rho*\omega_{\eta})(t,x)=\int_{x}^{x+\eta}\rho(t,y)\omega_{\eta}(y-x)\dd y.$$

On the on-ramp the idea is that at position $x$ the flow merging in the traffic way is inversely proportional to  the average density around position $x$, see Fig. \ref{fig:onofframp} , i.e, we write
\begin{equation}\label{nolreact1}
S_{\mathrm{on}}(t, x, \rho,\rho*\omega_{\eta,\delta}) =\mathbf{1}_{\mathrm{on}}(x) q_{\mathrm{on}}(t) (\rho_{\max}-\rho*\omega_{\eta,\delta}),
\end{equation}
with $$  (\rho*\omega_{\eta,\delta})(t,x)=\int_{x-\eta+\delta}^{x+\eta+\delta}\rho(t,y)\omega_{\eta,\delta}(y-x)\dd y,$$
with $\eta\in[0,1]$ and $\delta\in[-\eta,\eta]$. However, in the numerical test section we will see that the choice of the non-local term \eqref{nolreact1} does not guarantee that the proposed model satisfies a Maximum Principle, see Example 3. In order to overcome this difficulty, we consider a first variant of  \eqref{nolreact1} taking 
\begin{equation}\label{nolreact3}
S_{\mathrm{on}}(t, x, \rho,\rho*\omega_{\eta,\delta}) =\mathbf{1}_{\mathrm{on}}(x) q_{\mathrm{on}}(t) (\rho_{\max}-\rho)(\rho_{\max}-\rho*\omega_{\eta,\delta}).
\end{equation}
Note that this term contains a product which differentiates it from the original model. An alternative is to choose
\begin{equation}\label{nolreact2}
S_{\mathrm{on}}(t, x, \rho,\rho*\omega_{\eta,\delta}) =\mathbf{1}_{\mathrm{on}}(x) q_{\mathrm{on}}(t)\left(\rho_{\max}- \max\{\rho;\rho*\omega_{\eta,\delta}\}\right).
\end{equation}

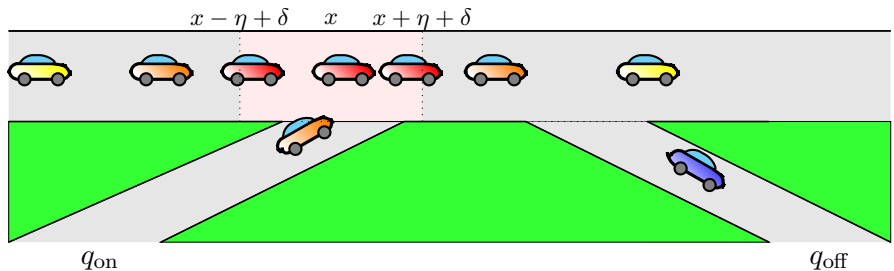
\begin{figure}[t]
 \centering
 \input{carplotnonloc.tex}
 \caption{Illustration of our model setting.}\label{fig:onofframp}
 \end{figure}
The purpose of this work is the study of the well-posedness of a nonlocal reaction traffic flow model with source term given by \eqref{nolreact3} and \eqref{nolreact2}.
\subsection{Related work} 
In \cite{bayen2020modeling, Blandin2016WellposednessOA, chiarello2020non, Ch_G_global, chiarello2019non,  GoatinScialanga, FKG2018} the authors studied a nonlocal conservation law to model vehicular traffic flow in the case $S_{\mathrm{on}}=S_{\mathrm{off}}=0$, i.e., without on- and off-ramps. 
The need to design more realistic models has led to the development of multi-lane vehicular traffic models among which we can highlight the following. In \cite{holden2019models}, it is introduced a new local model for multilane dense vehicular traffic by means of a system of a weakly coupled scalar conservation laws. In \cite{goatin2019multilane}, the authors consider the model proposed in \cite{holden2019models} but with a more general source terms and they allow for the presence of space discontinuities both in the speed law and in the number of lanes; in these two local models the source term accounts for the lane change rate and the key assumption is that the drivers prefer to drive faster, and that the tendency of a vehicle change the lines is proportional to the difference in velocity between neighboring lanes. In \cite{friedrich2020nonlocal} is studied a multilane model with local and non-local flux combined with a source term that also incorporates a nonlocality; here, the non-local source term describes the lane changing rate depending on a (nonlinear) evaluation of the velocity. In particular, the lane changing rate is proportional to the difference in the velocity between two adjacent lanes, but the velocities are evaluated in a neighbourhood of the current position, moreover, this rate is proportional also to the density in the receiving lane, meaning that if that lane is crowded only a few vehicles can actually change lane.\\  
Regarding to vehicular traffic flow models taking into account the presence of ramps we can mention \cite{liptak2021traffic}, where the authors study the (local) first order nonlinear conservation law \eqref{RTM}. In \cite{tie2009new} a (local) second order model is proposed to study the effects of on- and off-ramps on a main road traffic during two rush periods. Likewise, other works about the study of effects of ramps in vehicular traffic flow models are referenced in \cite{tie2009new}. In particular, in \cite{delle2014pde} the authors consider a Lighthill-Witham-Richards (LWR) traffic flow model on a junction composed by one mainline, an on-ramp and an off-ramp, which are connected by a node.  Moreover, in \cite{helbing2001master} a non-local gas-kinetic traffic model including ramps  is proposed, the model allows to simulate syncronized congested traffic and reproduces realistic phenomena of vehicular traffic by variations of the on-ramp flow . In \cite{liu1996modelling} a new modeling methodology for merging and diverging flows is studied, the methodology includes coupling effects between main and ramps flows and a new formulation for the modeling of traffic friction is also introduced.\\

\subsection{Outline of the paper} This work is organized as follows: In Section \ref{sec:model} we present the proposed mathematical model with all the considered assumptions on it. Afterwards, we introduce an upwind-type Scheme with two different source terms and derive important properties such as maximum principle, $\L{1}-$ bound and $\BV$ estimates. Furthermore, we derive the $\L{1}-$Lipschitz continuous dependence of solutions to \eqref{nonRTM} on the initial data and the terms $q_{\mathrm{on}}$ and $q_{\mathrm{off}}$ in Section \ref{sec:Existence}.  In Section \ref{Sec:Num_Exp}, we present numerical examples  illustrating the behavior of the solutions of our model.               

\section{Mathematical model }\label{sec:model} 
The main goal of this work is to study the well-posedness of the non-local reaction traffic model 
\begin{equation}\label{nonRTM}
    \rho_t+(\rho v(\rho*\omega_{\eta}))_x=S_{\mathrm{on}}(\cdot, \cdot, \rho,\rho*\omega_{\eta,\delta})-S_{\mathrm{off}}(\cdot, \cdot, \rho), \quad x\in\R,
\end{equation}
where $S_{\mathrm{on}}(\cdot, \cdot, \rho,\rho*\omega_{\eta,\delta})$ defined  in \eqref{nolreact3} or \eqref{nolreact2}, $S_{\mathrm{off}}$ defined by \eqref{soff_local} and initial condition 
\begin{eqnarray}\label{initial_condition}
 \rho(x,0)=\rho_0(x)\in \left(\L{1}\cap\BV\right)(\R,[0,1]).
\end{eqnarray}
From now on we called Model 0 the equations \eqref{nonRTM}-\eqref{nolreact1}-\eqref{initial_condition}, Model 1 the equations \eqref{nonRTM}-\eqref{nolreact3}-\eqref{initial_condition}, and Model 2 \eqref{nonRTM}-\eqref{nolreact2}-\eqref{initial_condition}. Let us assume the following assumptions:
\begin{equation}\label{Assumption1}\tag{H1}
\begin{array}{ll}
q_{\mathrm{on}}\in\L{\infty}(\R^{+};\R^{+}),q_{\mathrm{off}}\in\L{\infty}(\R^{+};\R^{+}).\\
v\in\C{2}(\R;[0,1]) \ v'\leq0.\\
 \omega_{\eta}\in\C{1}([0,\eta];\R^{+})\ \textup{with} \  \omega'_{\eta}(x)\leq0, \ \int_{0}^{\eta}\omega_{\eta}(x)\dd x=1, \ \forall\eta>0.\\
\omega_{\eta,\delta}\in\C{1}([\delta-\eta,\delta+\eta];\R^{+}) \ \textup{with} \ \omega'(x)_{\eta,\delta}\geq0 \ \text{for} \ x\in[\delta-\eta,0],\\
\omega'(x)_{\eta,\delta}\leq0 \ \text{ for } x\in[0,\delta+\eta],\ \text {and}\ \int_{\delta-\eta}^{\delta+\eta}\omega_{\eta,\delta}(x)\dd x=1,\ \forall\eta>0.
\end{array}
\end{equation}
We recall the definition of weak entropy solution for \eqref{nonRTM}.

\begin{definition}\label{weak_Sol}
Let $\rho_{0}\in (\L{1}\cap\BV)(\R;[0,1]).$ We say that $\rho\in\mathbf{C}([0,T]; \L{1}(\R;[0,1])),$ with $\rho(t,\cdot)\in\BV(\R;[0,1])$ for $t\in[0,T]$, is a weak solution to \eqref{nonRTM} with initial datum $\rho_{0}$ if for any $\varphi\in\Cc{1}([0,T[\times\R;\R) $
\begin{eqnarray*}
\int_{0}^{T}\int_{\R}\left(\rho\varphi_{t}+\rho V\varphi_{x}\right)\dd x\dd t+\int_{0}^{T}\int_{\Omega_{\mathrm{on}}}S_{\mathrm{on}}\varphi\dd x\dd t-\int_{0}^{T}\int_{\Omega_{\mathrm{off}}}S_{\mathrm{off}}\varphi\dd x\dd t+\int_{\R}\rho_{0}(x)\varphi(0,x)\dd x=0,
\end{eqnarray*}
where $V(t,x)=v((\rho*\omega)(t,x))$ and $S_{\mathrm{on}}$ is as in \eqref{nolreact3} or \eqref{nolreact2}.
\end{definition}
\begin{definition}\label{entropy_WS}
Let $\rho_{0}\in (\L{1}\cap\BV)(\R;[0,1]).$ We say that $\rho\in\mathbf{C}([0,T]; \L{1}(\R;[0,1])),$ with $\rho(t,\cdot)\in\BV(\R;[0,1])$ for $t\in[0,T]$, is a entropy weak solution to \eqref{nonRTM} with initial datum $\rho_{0}$ if for any $\varphi\in\Cc{1}([0,T[\times\R;\R)$ and for all $k\in\R$
\begin{eqnarray*}
\int_{0}^{T}\int_{\R}\left(\vert\rho-k \vert\varphi_{t}+\left|\rho-k\right| V\varphi_{x}-\sgn(\rho-k)k V_{x }\varphi\right)\dd x\dd t+\int_{0}^{T}\int_{\Omega_{\mathrm{on}}}\sgn(\rho-k)S_{\mathrm{on}}\varphi\dd x \dd t\\
-\int_{0}^{T}\int_{\Omega_{\mathrm{off}}}\sgn(\rho-k)S_{\mathrm{off}}\varphi\dd x\dd t+\int_{\R}\vert\rho_{0}-k \vert\varphi(0,x)\dd x\geq0.
\end{eqnarray*} 
\end{definition}
Our main result is given by the following theorem, which states the well-posedness of problem \eqref{nonRTM} to \eqref{initial_condition} with source term given by \eqref{nolreact3} or \eqref{nolreact2}.
\begin{theorem}\label{teo:uniq_exist}
Let $\rho_{0}\in\left(\L{1}\cap\BV\right)\left(\R;[0,1]\right)$. Assume $v\in\C{2}\left([0,1];\R\right).$ Then, for all $T>0$, the problem \eqref{nonRTM} has a unique solution $\rho\in\C{0}\left([0,T];\L{1}(\R;[0,1])\right)$ in the sense of Definition \ref{entropy_WS}. Moreover, the following estimates hold: for any $t\in[0,T]$
\begin{equation*}
\begin{array}{ll}
\left\|\rho(t)\right\|_{\L{1}(\R)}\leq \mathcal{R}_{1}(t),\\
0\leq\rho(t,x)\leq1,\\
TV(\rho(t))\leq e^{t\mathcal{H}}\left(TV(\rho_{0})+t\left(\frac{\left\|q_{\mathrm{on}}\right\|_{\L{\infty}([0,T])}+\left\|q_{\mathrm{off}}\right\|_{\L{\infty}([0,T])}}{L}\right)\right),
\end{array}
\end{equation*}
where 
\begin{eqnarray}
\label{R_1}\mathcal{R}_{1}&=&\left\|\rho_{0}\right\|_{\L{1}(\R)}+\left\|q_{\mathrm{on}(\cdot)}\right\|_{\L{1}([0,t])}-\min_{x\in\Omega_{\mathrm{on}}}\left\|q_{\mathrm{on}}(\cdot)\rho(\cdot,x)\right\|_{\L{1}([0,t])}\\
&&-\min_{x\in\Omega_{\mathrm{off}}}\left\|q_{\mathrm{off}}(\cdot)\rho(\cdot,x)\right\|_{\L{1}([0,t])},\nonumber\\
\label{H}\mathcal{H}&=&2\left\|q_{\mathrm{on}}\right\|_{\L{\infty}([0,T])}+\left\|q_{\mathrm{off}}\right\|_{\L{\infty}([0,T])}+\omega_{\eta}(0)\mathcal{L}\\
\label{L}\mathcal{L}&=&\left(\|v\|_{\L{\infty}([0,1])}+\|v'\|_{\L{\infty}([0,1])}\right).
\end{eqnarray}

\end{theorem}

\section{Existence of entropy solution}\label{sec:Existence}
\subsection{Numerical discretization}
We take a space step $\Delta x$ such that $\eta=N\Delta x$, for some $N\in\N$, and a time step $\Delta t$ subject to a CFL condition which will be specified later. For any $j\in\Z$, let $x_{j-1/2}=j\Delta x$ be a cells interfaces, $x_{j}=\bigg(j+\frac{1}{2}\bigg)\Delta x$ the cells centers. We consider ramps with length $L$ and take $L=\ell\Delta x, \ \textup{for some} \  \ell\in\Z^{+}$ such that $\underline{x}_{\mathrm{on}}=x_{\underline{k}_{\mathrm{on}}+1/2}$, $\overline{x}_{\mathrm{on}}=x_{\underline{k}_{\mathrm{on}}+1/2+\ell}$, $\underline{x}_{\mathrm{off}}=x_{\underline{k}_{\mathrm{off}}+1/2}$ and $\overline{x}_{\mathrm{off}}=x_{\underline{k}_{\mathrm{off}}+1/2+\ell}$, for some $\underline{k}_{\mathrm{on}}, \underline{k}_{\mathrm{off}}\in\Z$. With this notation, we define the subdomains $\Omega_{\mathrm{on}}=[\underline{x}_{\mathrm{on}},\overline{x}_{\mathrm{on}}]$, $\Omega_{\mathrm{off}}=[\underline{x}_{\mathrm{off}},\overline{x}_{\mathrm{off}}]$, and we put $\Omega_{\mathrm{on}}^{k}=[\underline{k}_{\mathrm{on}}+1,\underline{k}_{\mathrm{on}}+\ell]$ and $\Omega_{\mathrm{off}}^{k}=[\underline{k}_{\mathrm{off}}+1,\underline{k}_{\mathrm{off}}+\ell]$.\\
We fix $T>0$, and set $N_{T}\in\N$ such that $N_{T}\Delta t\leq T<\left(N_{T}+1\right)\Delta t$ and define the time mesh as $t^{n}=n\Delta t$ for $n=0,\ldots,N_{T}$. Set $\lambda=\Delta t/\Delta x$. The initial data is approximated for $j\in\Z,$ as follows:
\begin{eqnarray*}
\rho_{j}^{0}=\frac{1}{\Delta x}\int_{x_{j-1/2}}^{x_{j+1/2}}\rho_{0}(x)\dd x.
\end{eqnarray*}
We define a piecewise constant approximate solution $\rho_{\Delta}(t,x)$ to \eqref{nonRTM} as
\begin{eqnarray}\label{Sln_approx_tramos}
\rho_{\Delta}(t,x)=\rho_j^n,\quad \text{ for } 
\begin{cases}
t\in\left[t^n,t^{n+1}\right[ \\
x\in]x_{j-1/2},x_{j+1/2}],
\end{cases}\quad \text{where } \begin{array}{lcc}
             n=0,\dots,N_{T}-1, \\
             j\in\Z.
             \end{array} \
\end{eqnarray}
The $S_{\mathrm{on}}$ terms \eqref{nolreact3} and \eqref{nolreact2} are discretized via 
\begin{eqnarray}\label{Son_def1}
S_{\mathrm{on}}\left(t^{n+1/2},x_{j},q_{\mathrm{on}},\rho_{j}^{n+1/2},R_{\mathrm{on},j}^{n+1/2}\right)&=&\mathbf{1}_{\mathrm{on},j}q_{\mathrm{on}}^{n+1/2}(1-\rho_{j}^{n+1/2})(1-R_{\mathrm{on},j}^{n+1/2}),\\
\label{Son_def2}S_{\mathrm{on}}\left(t^{n+1/2},x_{j},q_{\mathrm{on}},\rho_{j}^{n+1/2},R_{\mathrm{on},j}^{n+1/2}\right)&=&\mathbf{1}_{\mathrm{on},j}q_{\mathrm{on},j}^{n+1/2}\left(1-\max\left\{\rho_{j}^{n+1/2},R_{\mathrm{on},j}^{n+1/2}\right\}\right).
\end{eqnarray}
The $S_{\mathrm{off}}$ term is discretizated via
\begin{eqnarray}
\label{Soff_def}S_{\mathrm{off}}\left(t^{n+1/2},x_{j},q_{\mathrm{off}},\rho_{j}^{n+1/2}\right)&=&\mathbf{1}_{\mathrm{off},j}q_{\mathrm{off}}^{n+1/2}\rho_{j}^{n+1/2},
\end{eqnarray}
where we denote 
\begin{eqnarray*}
\mathbf{1}_{\mathrm{on},j}&=&
 \begin{cases}
    \frac{1}{\Delta x}\int_{x_{j-1/2}}^{x_{j+1/2}}\mathbf{1}_{\mathrm{on}}(x)\dd x,\quad &\underline{x}_{\mathrm{on},k}\leq x_{j}\leq \overline{x}_{\mathrm{on},k},\\
    0 \quad&\text{otherwise}.
    \end{cases}
    \end{eqnarray*}
    \begin{eqnarray*}
    \mathbf{1}_{\mathrm{off},j}=
    \begin{cases}
     \frac{1}{\Delta x}\int_{x_{j-1/2}}^{x_{j+1/2}}\mathbf{1}_{\mathrm{off}}(x)\dd x,\quad &\underline{x}_{\mathrm{off},k}\leq x_{j}\leq \overline{x}_{\mathrm{off},k},\\
    0 \quad&\text{otherwise}.    
    \end{cases}
\end{eqnarray*}
\begin{eqnarray*}
 q_{\mathrm{on}}^{n+1/2}&=&\frac{1}{\Delta t}\int_{t^{n}}^{t^{n+1}}q_{\mathrm{on}}(t)\dd t,\qquad
 q_{\mathrm{off}}^{n+1/2}=\frac{1}{\Delta t}\int_{t^{n}}^{t^{n+1}}q_{\mathrm{off}}(t)\dd t,
\end{eqnarray*}

The approximate solution $\rho_{\Delta}$ is obtained via an upwind-type scheme together with operator splitting to account for the reaction term, see \textbf{Algorithm \ref{Algorithm_Scheme}}

\begin{Alg}[Upwind scheme]  \label{Algorithm_Scheme}  \hfill 
\begin{itemize}
\item[] { Input: approximate solution vector $\{\rho_j^n\}_{j\in\Z}$  for $t=t^n$}  
 \item[] {\bf do}  $j\in\Z$  
\begin{align}\label{rhon1/2}\rho_j^{n+1/2}  \leftarrow \rho_j^n - \lambda 
  \bigl( \rho_j^n v(R_{j+1/2}^{n})- \rho_{j-1}^n v(R_{j-1/2}^{n}) \bigr)
  \end{align}
   \item[] {\bf enddo} 
   \item[] {\bf do} $j\in\Z$ 
  
  \item[] $S_{\mathrm{on},j}^{n+1/2}  \leftarrow   S_{\mathrm{on}}\left(t^{n+1/2},x_{j},\rho_{j}^{n+1/2},R_{\mathrm{on},j}^{n+1/2}\right), $ using \eqref{Son_def1} or \eqref{Son_def2}, 
\item[] $S_{\mathrm{off},j}^{n+1/2}  \leftarrow   S_{\mathrm{off}}\left(t^{n+1/2},x_{j},\rho_{j}^{n+1/2}\right),$ using \eqref{Soff_def},   
  \begin{align}\label{Source_disc}\rho_{j}^{n+1}  \leftarrow  \rho_{j}^{n+1/2} +\Delta t S_{\mathrm{on},j}^{n+1/2}- \Delta t S_{\mathrm{off},j}^{n+1/2}
  \end{align}
  \item[] {\bf enddo} 
  \item[] { Output:  approximate solution vector $\{\rho_j^{n+1}\}_{j\in\Z}$ for $t=t^{n+1}=t^n+\Delta t$.}   
  \end{itemize} 
\end{Alg}  
The terms $R_{j+1/2}^{n}$, $R_{\mathrm{on},j}^{n+1/2}$ for $j\in\Z$ and $n=0,\ldots,N_{T}-1$ denotes the discrete convolution operators in the velocity and source term and they are defined, respectively, by 
\begin{eqnarray*}
R_{j+1/2}^{n}&=&\sum_{p=0}^{\lfloor \eta/\Delta x \rfloor-1}\gamma_{p}\rho_{j+p+1}^{n},\\
R_{\mathrm{on},j}^{n+1/2}&=&\sum_{h=\lfloor \frac{\delta-\eta}{\Delta x} \rfloor}^{\lfloor \frac{\delta+\eta}{\Delta x} \rfloor-1}\Hat{\gamma}_{h}\rho_{j+h}^{n+1/2}.
\end{eqnarray*}
Here we denote $\gamma_{p}=\int_{x_{p-1/2}}^{x_{p+1/2}}\omega_{\eta}(y-x)\dd y$, for $p\in[0,\lfloor \eta/\Delta x \rfloor-1]$
and $\Hat{\gamma}_{h}=\int_{x_{h-1/2}}^{x_{h+1/2}}\omega_{\eta,\delta}(y-x)\dd y$, for $h\in[\lfloor (\delta-\eta)/\Delta x \rfloor,\lfloor (\delta+\eta)/\Delta x \rfloor-1]$.

\begin{remark}\label{remark:Bound_R}
If \  $0\leq\rho_{j}^{n+1/2}\leq 1$ for all $j\in\Z$, then for all $n\in\{0,\ldots,N_{T}-1\}$,\\ $\left\|R_{\mathrm{on}}^{n+1/2}\right\|_{\L{\infty}(\Omega_{\mathrm{on}}^{k})}\leq1$. Indeed, we have that
\begin{eqnarray*}
\left|R_{\mathrm{on},j}^{n+1/2}\right|&\leq&\sum_{h=\lfloor \frac{\delta-\eta}{\Delta x} \rfloor}^{\lfloor \frac{\delta+\eta}{\Delta x} \rfloor-1}\Hat{\gamma}_{h}\left|\rho_{j+h+1}^{n+1/2}\right|
\leq\sum_{h=\lfloor \frac{\delta-\eta}{\Delta x} \rfloor}^{\lfloor \frac{\delta+\eta}{\Delta x} \rfloor-1}\Hat{\gamma}_{h}=1.
\end{eqnarray*}
\end{remark}
\begin{remark}\label{remark:bound_sum_Ron}
The discrete convolution operator $R_{\mathrm{on},j}^{n+1/2}$ satisfies 
$$\sum_{j\in\Z}\left|R_{\mathrm{on},j+1}^{n+1/2}-R_{\mathrm{on},j}^{n+1/2}\right|\leq\sum_{j\in\Z}\left|\rho_{j+1}^{n+1/2}-\rho_{j}^{n+1/2}\right|.$$
The proof of this property can be seen in \cite{friedrich2020nonlocal} Lemma 3.2.
\end{remark}

\subsection{Existence of solution Model 1} In order to prove the existence of solution of model \eqref{nonRTM}-\eqref{nolreact3}, in the next lemmas we will show some properties of the approximate solutions constructed by the \textbf{Algorithm \ref{Algorithm_Scheme}}.
\begin{lemma}[Maximum principle]\label{Lemma:Maximum_pple}
Let $\rho_{0}\in\L{\infty}(\R;[0,1])$. Let hypotheses \eqref{Assumption1} and the following Courant-Friedrichs-Levy (CFL) condition hold
\begin{eqnarray}\label{CFL}
\Delta t\leq \min\left\{\frac{\Delta x}{\left(\gamma_{0}\|v'\|_{\L{\infty}([0,1])}+\|v\|_{\L{\infty}([0,1])}\right)},\frac{L}{\left\|q_{\mathrm{on}}\right\|_{\L{\infty}([0,T])}+\left\|q_{\mathrm{off}}\right\|_{\L{\infty}([0,T])}}\right\}
\end{eqnarray}
then for all $t>0$ and $x\in\R$ the piece-wise constant approximate solution $\rho_{\Delta}$ constructed through \textbf{Algorithm \ref{Algorithm_Scheme}} is such that
\begin{eqnarray*}
0\leq \rho_{\Delta}(t,x)\leq1.    
\end{eqnarray*}
\end{lemma}
\begin{proof}
The proof is made by induction. Let us assume that $0\leq\rho_{j}^{n}\leq1$ for all $j\in\Z$. Consider the convective step \eqref{rhon1/2} of \textbf{Algorithm \ref{Algorithm_Scheme}}, by CFL condition \eqref{CFL} we have $0\leq\rho_{j}^{n+1/2}\leq 1$ for $j\in\Z$ (see Theorem 3.3 of \cite{FKG2018}).\\

Now focus on the remaining step, involving the source term. 

\begin{eqnarray*}
\rho_{j}^{n+1}&=&\rho_{j}^{n+1/2}+\Delta t\left( \mathbf{1}_{\mathrm{on},j}q_{\mathrm{on}}^{n+1/2}\big(1-\rho_{j}^{n+1/2}\big)\left(1-R_{\mathrm{on},j}^{n+1/2}\right)-\mathbf{1}_{\mathrm{off},j}q_{\mathrm{off}}^{n+1/2}\rho_{j}^{n+1/2}\right)\nonumber\\
&\leq&\rho_{j}^{n+1/2}+\Delta t \mathbf{1}_{\mathrm{on},j}q_{\mathrm{on}}^{n+1/2}\left(1-\rho_{j}^{n+1/2}\right)-\Delta t\mathbf{1}_{\mathrm{off},j}q_{\mathrm{off}}^{n+1/2}\rho_{j}^{n+1/2}\nonumber\\
&=&\left(1-\Delta t\left(\mathbf{1}_{\mathrm{on},j}q_{\mathrm{on}}^{n+1/2}+\mathbf{1}_{\mathrm{off},j}q_{\mathrm{off}}^{n+1/2}\right)\right)\rho_{j}^{n+1/2}+\Delta t\mathbf{1}_{\mathrm{on},j}q_{\mathrm{on}}^{n+1/2}.
\end{eqnarray*}
Because of CFL condition \eqref{CFL}, the last right-hand side is a convex combination of $\rho_{j}^{n+1/2}$ and one. Then $\rho_{j}^{n+1}\in\left[\rho_{j}^{n+1/2},1\right]$ and since $\rho_{j}^{n+1/2}\in[0,1]$, we therefore conclude that $0\leq\rho_{j}^{n+1}\leq1$, for $j\in\Z.$\\  
\end{proof}

\begin{lemma}[$\L{1}-Bound$]\label{Lemma:L1_norm}
Let $\rho_{0}\in\L{1}(\R,[0,1])$. Let \eqref{Assumption1} and the CFL condition \eqref{CFL} hold. Then, the piece-wise constant approximate solution $\rho_{\Delta}$ constructed through \textbf{Algorithm \ref{Algorithm_Scheme}} satisfies, for all $T>0$,
\begin{eqnarray*}
\left\|\rho_{\Delta}(T,\cdot)\right\|_{\L{1}(\R)}\leq\mathcal{C}_{1}(T),
\end{eqnarray*}
with
\begin{equation}\label{C_1}
\mathcal{C}_{1}(t)=\left\|\rho_{0}\right\|_{\L{1}(\R)}+\left\|q_{\mathrm{on}}\right\|_{\L{1}([0,t])}-\min_{x\in\Omega_{\mathrm{on}}}\left\|q_{\mathrm{on}}(\cdot)\rho_{\Delta}(\cdot,x)\right\|_{\L{1}([0,t])}-\min_{x\in\Omega_{\mathrm{off}}}\left\|q_{\mathrm{off}}(\cdot)\rho_{\Delta}(\cdot,x)\right\|_{\L{1}([0,t])}.
\end{equation}
\end{lemma}

\begin{proof}
For the conservative form of the scheme \eqref{rhon1/2}, it is satisfied
\begin{eqnarray*}
\left\|\rho^{n+1/2}\right\|_{\L{1}(\R)}=\left\|\rho^{n}\right\|_{\L{1}(\R)}.
\end{eqnarray*}
Now, we going to work $\L{1}$ norm for relaxation step \eqref{Source_disc}. By Remark \ref{remark:Bound_R} and CFL condition \eqref{CFL} we have
\begin{eqnarray}\label{abs_rhon+1}
\left|\rho_{j}^{n+1}\right|\leq\left|\rho_{j}^{n+1/2}\right|+\Delta t\mathbf{1}_{\mathrm{on},j}q_{\mathrm{on}}^{n+1/2}\left(1-\left|\rho_{j}^{n+1/2}\right|\right)-\Delta t \mathbf{1}_{\mathrm{off},j}q_{\mathrm{off}}^{n+1/2}\left|\rho_{j}^{n+1/2}\right|,
\end{eqnarray}
multiplying this inequality by $\Delta x$ and summing over all $j\in\Z$ we obtain
\begin{eqnarray*}
\left\|\rho^{n+1}\right\|_{\L{1}(\R)}&\leq&\left\|\rho^{n+1/2}\right\|_{\L{1}(\R)}+\Delta tq_{\mathrm{on}}^{n+1/2}\left(\Delta x\sum_{j\in\Omega_{\mathrm{on}}^{k}}\mathbf{1}_{\mathrm{on},j}-\Delta x\sum_{j\in\Omega_{\mathrm{on}}^{k}}\mathbf{1}_{\mathrm{on},j}\left|\rho_{j}^{n+1/2}\right|\right)\\
&&-\Delta t q_{\mathrm{off}}^{n+1/2}\Delta x\sum_{j\in\Omega_{\mathrm{off}}^{k}}\mathbf{1}_{\mathrm{off},j}\left|\rho_{j}^{n+1/2}\right|\\
&=&\left\|\rho^{n+1/2}\right\|_{\L{1}(\R)}+\Delta tq_{\mathrm{on}}^{n+1/2}\left(1-\frac{\left\|\rho^{n+1/2}\right\|_{\L{1}(\Omega_{\mathrm{on}}^{k})}}{L}\right)\\
&&-\Delta t q_{\mathrm{off}}^{n+1/2}\frac{\left\|\rho^{n+1/2}\right\|_{\L{1}(\Omega_{\mathrm{off}}^{k})}}{L}\\
&\leq&\left\|\rho^{n}\right\|_{\L{1}(\R)}+\Delta tq_{\mathrm{on}}^{n+1/2}\left(1-\min_{j\in\Omega_{\mathrm{on}}^{k}}\rho^{n+1/2}\right)\\
&&-\Delta t q_{\mathrm{off}}^{n+1/2}\min_{j\in\Omega_{\mathrm{off}}^{k}}\rho_{j}^{n+1/2}\\
&=&\left\|\rho^{n}\right\|_{\L{1}(\R)}+\Delta tq_{\mathrm{on}}^{n+1/2}-\Delta t\min_{j\in\Omega_{\mathrm{on}}^{k}}q_{\mathrm{on}}^{n+1/2}\rho_{j}^{n+1/2}\\
&&-\Delta t \min_{j\in\Omega_{\mathrm{off}}^{k}}q_{\mathrm{off}}^{n+1/2}\rho_{j}^{n+1/2}.\\
\end{eqnarray*}
Thus, by a standard iterative procedure we can deduce 
\begin{eqnarray*}
\left\|\rho^{n}\right\|_{\L{1}(\R)}\leq\left\|\rho_{0}\right\|_{\L{1}(\R)}+\left\|q_{\mathrm{on}}\right\|_{\L{1}([0,T])}-\min_{x\in\Omega_{\mathrm{on}}}\left\|q_{\mathrm{on}}(\cdot)\rho_{\Delta}(\cdot,x)\right\|_{\L{1}([0,T])}-\min_{x\in\Omega_{\mathrm{off}}}\left\|q_{\mathrm{off}}(\cdot)\rho_{\Delta}(\cdot,x)\right\|_{\L{1}([0,T])}.
\end{eqnarray*}
\end{proof}

\subsection{BV estimates}
\

We first prove the Lipschitz continuity of the source terms \eqref{Son_def1} in its second, third and fourth argument and \eqref{Soff_def} in its second and third argument.
\begin{lemma}\label{Lipschitz_Son_Soff}
The map $S_{\mathrm{on}}$ defined in \eqref{Son_def1} is Lipschitz continuous in second, third and fourth argument with Lipschitz constant $\left\|q_{\mathrm{on}}\right\|_{\L{\infty}([0,T])}$, and the map $S_{\mathrm{off}}$ defined in \eqref{Soff_def} is Lipschitz continuous in second and third argument with Lipschitz constant $\left\|q_{\mathrm{off}}\right\|_{\L{\infty}([0,T])}$.
\end{lemma}
\begin{proof}
Let us start with term \eqref{Son_def1}. We denote $\mathcal{S}_{\mathrm{on}}=S_{\mathrm{on}}(t,x,\rho,R_{\mathrm{on}})-S_{\mathrm{on}}(t,\tilde{x},\tilde{\rho},\tilde{R}_{\mathrm{on}})$, then
\begin{eqnarray*}
\left|\mathcal{S}_{\mathrm{on}}\right|&\leq&\left|S_{\mathrm{on}}(t,x,\rho,R_{\mathrm{on}})-S_{\mathrm{on}}(t,x,\tilde{\rho},R_{\mathrm{on}})\right|\\
&&+\left|S_{\mathrm{on}}(t,x,\tilde{\rho},R_{\mathrm{on}})-S_{\mathrm{on}}(t,{x},\tilde{\rho},\tilde{R}_{\mathrm{on}})\right|\\
&&+\left|S_{\mathrm{on}}(t,x,\tilde{\rho},\tilde{R}_{\mathrm{on}})-S_{\mathrm{on}}(t,\tilde{x},\tilde{\rho},\tilde{R}_{\mathrm{on}})\right|\\
&=&\left|\mathbf{1}_{\mathrm{on}}q_{\mathrm{on}}\left(1-R_{\mathrm{on}}\right)\left(\tilde{\rho}-\rho\right)\right|+\left|\mathbf{1}_{\mathrm{on}}q_{\mathrm{on}}\left(1-\tilde{\rho}\right)\left(\tilde{R}_{\mathrm{on}}-R_{\mathrm{on}}\right)\right|\\
&&+\left|\left(\mathbf{1}_{\mathrm{on}}-\tilde{\mathbf{1}}_{\mathrm{on}}\right)q_{\mathrm{on}}\left(1-\tilde{\rho}\right)\left(1-\tilde{R}_{\mathrm{on}}\right)\right|\\
&\leq&\left\|q_{\mathrm{on}}\right\|_{\L{\infty}([0,T])}\left|\tilde{\rho}-\rho\right|+\left\|q_{\mathrm{on}}\right\|_{\L{\infty}([0,T])}\left|\tilde{R}_{\mathrm{on}}-R_{\mathrm{on}}\right|\\
&&+\left\|q_{\mathrm{on}}\right\|_{\L{\infty}([0,T])}\left|\mathbf{1}_{\mathrm{on}}-\tilde{\mathbf{1}}_{\mathrm{on}}\right|\\
&\leq&\left\|q_{\mathrm{on}}\right\|_{\L{\infty}([0,T])}\left(\left|\tilde{\rho}-\rho\right|+\left|\tilde{R}_{\mathrm{on}}-R_{\mathrm{on}}\right|+\left|\mathbf{1}_{\mathrm{on}}-\tilde{\mathbf{1}}_{\mathrm{on}}\right|\right).
\end{eqnarray*}
Now, we prove the Lipschitz continuity of $S_{\mathrm{off}}$ term \eqref{Soff_def}. Denoting\\ $\mathcal{S}_{\mathrm{off}}=S_{\mathrm{off}}(t,x,\rho)-S_{\mathrm{off}}(t,\tilde{x},q_{\mathrm{off}},\tilde{\rho}),$ we get
\begin{eqnarray*}
\left|\mathcal{S}_{\mathrm{off}}\right|&\leq&\left|S_{\mathrm{off}}(t,x,\rho)-S_{\mathrm{off}}(t,\tilde{x},\rho,)\right|+\left|S_{\mathrm{off}}(t,\tilde{x},\rho)-S_{\mathrm{off}}(t,\tilde{x},\tilde{\rho})\right|\\
&=&\left|\mathbf{1}_{\mathrm{off}}q_{\mathrm{off}}\rho-\tilde{\mathbf{1}}_{\mathrm{off}}q_{\mathrm{off}}\rho\right|+\left|\tilde{\mathbf{1}}_{\mathrm{off}}q_{\mathrm{off}}\rho-\tilde{\mathbf{1}}_{\mathrm{off}}q_{\mathrm{off}}\tilde{\rho}\right|\\
&\leq&\left\|q_{\mathrm{off}}\right\|_{\L{\infty}([0,T])}\left(\left|\mathbf{1}_{\mathrm{off}}-\tilde{\mathbf{1}}_{\mathrm{off}}\right|+\left|\rho-\tilde{\rho}\right|\right),
\end{eqnarray*}

Thus, we have completed the proof. 
\end{proof}
The Lipschitz continuity of the source term proved in Lemma \ref{Lipschitz_Son_Soff} is one of the key ingredients in order to prove the following total variation bound on the numerical approximation. \\

\begin{proposition}[$\BV$ estimate in space]\label{prop:BV_Space} 
Let $\rho_{0}\in\left(\L{1}\cap\BV\right)\left(\R;[0,1]\right).$ Assume that the hypotheses \eqref{Assumption1} and CFL condition \eqref{CFL} hold. Then, for $n=0,\ldots,N_{T}-1$ the following estimate holds
\begin{eqnarray*}
\sum_{j\in\Z}\left|\rho_{j+1}^{n}-\rho_{j}^{n}\right|\leq e^{T\mathcal{H}}\left(TV(\rho_{0})+T\left(\frac{\left\|q_{\mathrm{on}}\right\|_{\L{\infty}([0,T])}+\left\|q_{\mathrm{off}}\right\|_{\L{\infty}([0,T])}}{L}\right)\right),
\end{eqnarray*}
with $\mathcal{H}$ like in \eqref{H}
\end{proposition}
\begin{proof} Let us compute
\begin{eqnarray*}
\rho_{j+1}^{n+1}-\rho_{j}^{n+1}&=&\rho_{j+1}^{n+1/2}-\rho_{j}^{n+1/2}
+\Delta t\left[S_{\mathrm{on},j+1}^{n+1/2}
-S_{\mathrm{on},j}^{n+1/2}\right]\\
&&-\Delta t\left[S_{\mathrm{off},j+1}^{n+1/2}
-S_{\mathrm{off},j}^{n+1/2}\right]
\end{eqnarray*}

By the Lipschitz continuity of the source term proved in Lemma \ref{Lipschitz_Son_Soff} and the property of the discrete convolution operator given in Remark \ref{remark:bound_sum_Ron}, we get
\begin{eqnarray}
\sum_{j\in\Z}\left|\rho_{j+1}^{n+1}-\rho_{j}^{n+1}\right|&\leq&\left(1+\frac{\Delta t}{L}\left\|q_{\mathrm{on}}\right\|_{\L{\infty}([0,T])}\right)\sum_{j\in\Z}\left|\rho_{j+1}^{n+1/2}-\rho_{j}^{n+1/2}\right|\nonumber\\
&&+\Delta t\left\|q_{\mathrm{on}}\right\|_{\L{\infty}([0,T])}\sum_{j\in\Omega_{\mathrm{on}}^{k}}\left|\mathbf{1}_{\mathrm{on},j+1}-\mathbf{1}_{\mathrm{on},j}\right|\nonumber\\
&&+\Delta t\left\|q_{\mathrm{on}}\right\|_{\L{\infty}([0,T])}\sum_{j\in\Z}\left|R_{\mathrm{on},j+1}^{n+1/2}-R_{\mathrm{on},j}^{n+1/2}\right|\nonumber\\
&&+\frac{\Delta t}{L}\left\|q_{\mathrm{off}}\right\|_{\L{\infty}([0,T])}\sum_{j\in\Z}\left|\rho_{j+1}^{n+1/2}-\rho_{j}^{n+1/2}\right|\nonumber\\
&&+\frac{\Delta t}{L}\left\|q_{\mathrm{off}}\right\|_{\L{\infty}([0,T])}\sum_{j\in\Omega_{\mathrm{off}}}\left|\mathbf{1}_{\mathrm{off},j+1}-\mathbf{1}_{\mathrm{off},j}\right|\nonumber\\
&\leq&\left(1+\frac{\Delta t}{L}\left(2\left\|q_{\mathrm{on}}\right\|_{\L{\infty}([0,T])}+\left\|q_{\mathrm{off}}\right\|_{\L{\infty}([0,T])}\right)\right)\sum_{j\in\Z}\left|\rho_{j+1}^{n+1/2}-\rho_{j}^{n+1/2}\right|\nonumber\\
&&+\Delta t\left\|q_{\mathrm{on}}\right\|_{\L{\infty}([0,T])}\sum_{j\in\Omega_{\mathrm{on}}^{k}}\left|\mathbf{1}_{\mathrm{on},j+1}-\mathbf{1}_{\mathrm{on},j}\right|\nonumber\\
&&+\Delta t\left\|q_{\mathrm{off}}\right\|_{\L{\infty}([0,T])}\sum_{j\in\Omega_{\mathrm{off}}}\left|\mathbf{1}_{\mathrm{off},j+1}-\mathbf{1}_{\mathrm{off},j}\right|\nonumber\\
&\leq&\label{BV_react}\left(1+\frac{\Delta t}{L}\left(2\left\|q_{\mathrm{on}}\right\|_{\L{\infty}([0,T])}+\left\|q_{\mathrm{off}}\right\|_{\L{\infty}([0,T])}\right)\right)\sum_{j\in\Z}\left|\rho_{j+1}^{n+1/2}-\rho_{j}^{n+1/2}\right|\nonumber\\
&&+\Delta t\left(\frac{\left\|q_{\mathrm{on}}\right\|_{\L{\infty}([0,T])}+\left\|q_{\mathrm{off}}\right\|_{\L{\infty}([0,T])}}{L}\right).
\end{eqnarray}

Now, for convective part \eqref{rhon1/2} we follow \cite{FKG2018} and get
\begin{eqnarray*}
\left|\rho_{j+1}^{n+1/2}-\rho_{j}^{n+1/2}\right|&\leq&\left(1+\Delta t\omega_{\eta}(0)\mathcal{L}\right)\sum_{j\in\Z}\left|\rho_{j+1}^{n}-\rho_{j}^{n}\right|,
 \end{eqnarray*}
 with $\mathcal{L}=\left(\|v\|_{\L{\infty}([0,1])}+\|v'\|_{\L{\infty}([0,1])}\right)$.\\
 Plugging the inequality above in \eqref{BV_react} we obtain
  \begin{eqnarray*}
\sum_{j\in\Z}\left|\rho_{j+1}^{n+1}-\rho_{j}^{n+1}\right|&\leq&\left(1+\frac{\Delta t}{L}\left(2\left\|q_{\mathrm{on}}\right\|_{\L{\infty}([0,T])}+\left\|q_{\mathrm{off}}\right\|_{\L{\infty}([0,T])}\right)\right)\left(1+\Delta t\omega_{\eta}(0)\mathcal{L}\right)\sum_{j\in\Z}\left|\rho_{j+1}^{n}-\rho_{j}^{n}\right|\\
 &&+\Delta t\left(\frac{\left\|q_{\mathrm{on}}\right\|_{\L{\infty}([0,T])}+\left\|q_{\mathrm{off}}\right\|_{\L{\infty}([0,T])}}{L}\right),
 \end{eqnarray*}

which applied recursively yields

\begin{eqnarray}\label{BVSpace_Bound}
\sum_{j\in\Z}\left|\rho_{j+1}^{n}-\rho_{j}^{n}\right|
&\leq&e^{T\mathcal{H}}\left(TV(\rho_{0})+T\left(\frac{\left\|q_{\mathrm{on}}\right\|_{\L{\infty}([0,T])}+\left\|q_{\mathrm{off}}\right\|_{\L{\infty}([0,T])}}{L}\right)\right),
\end{eqnarray}
with $\mathcal{H}=\frac{1}{L}\left(2\left\|q_{\mathrm{on}}\right\|_{\L{\infty}([0,T])}+\left\|q_{\mathrm{off}}\right\|_{\L{\infty}([0,T])}\right)+\omega_{\eta}(0)\mathcal{L}$ .

\end{proof}

\begin{proposition}[$\BV$ estimate in space and time]\label{prop:BV_SpaceTime}
Let hypotheses \eqref{Assumption1} hold,\\ $\rho_{0}\in\left(\L{1}\cap\BV\right)\left(\R;[0,1]\right)$. If the CFL condition \eqref{CFL} holds, then, for every $T>0$ the following discrete space and time total variation estimate is satisfied:  
\begin{eqnarray*}
TV(\rho_{\Delta};[0,T]\times\R)
&\leq& T\mathcal{C}_{xt}(T),
\end{eqnarray*}
with 
\begin{eqnarray}
\mathcal{C}_{xt}(T)&=&e^{T\mathcal{H}}\left(\left(1+2\mathcal{L}\right)\left(TV(\rho_{0})+T\left(\frac{\left\|q_{\mathrm{on}}\right\|_{\L{\infty}([0,T])}+\left\|q_{\mathrm{off}}\right\|_{\L{\infty}([0,T])}}{L}\right)\right)\right)\nonumber\\
&&\label{C_xt}+\left(\frac{\left\|q_{\mathrm{on}}\right\|_{\L{\infty}([0,T])}+\left\|q_{\mathrm{off}}\right\|_{\L{\infty}([0,T])}}{L}\right)\mathcal{C}_{1}(T)+\frac{\left\|q_{\mathrm{on}}\right\|_{\L{\infty}([0,T])}}{L}.
\end{eqnarray}
\end{proposition}
\begin{proof}
\begin{eqnarray*}
TV(\rho_{\Delta};[0,T]\times\R)&=&\sum_{n=0}^{N_{T}-1}\sum_{j\in\Z}\Delta t\left|\rho_{j+1}^{n}-\rho_{j}^{n}\right|+\left(T-N_{T}\Delta t\right)\sum_{j\in\Z}\left|\rho_{j+1}^{N_{T}}-\rho_{j}^{N_{T}}\right|\\
&&+\sum_{n=0}^{N_{T}-1}\sum_{j\in\Z}\Delta x\left|\rho_{j}^{n+1}-\rho_{j}^{n}\right|.
\end{eqnarray*}
By $\BV$ estimate in space \eqref{BVSpace_Bound}, we have

\begin{eqnarray}\label{BV_Space_time1}
&&\sum_{n=0}^{N_{T}-1}\sum_{j\in\Z}\Delta t\left|\rho_{j+1}^{n}-\rho_{j}^{n}\right|+\left(T-N_{T}\Delta t\right)\sum_{j\in\Z}\left|\rho_{j+1}^{N_{T}}-\rho_{j}^{N_{T}}\right|\nonumber\\
&\leq&Te^{T\mathcal{H}}\left(TV(\rho_{0})+T\left(\frac{\left\|q_{\mathrm{on}}\right\|_{\L{\infty}([0,T])}+\left\|q_{\mathrm{off}}\right\|_{\L{\infty}([0,T])}}{L}\right)\right).
\end{eqnarray}

On the other hand, observe that
\begin{eqnarray}\label{rhom+1_rhom}
\left|\rho_{j}^{n+1}-\rho_{j}^{n}\right|\leq\left|\rho_{j}^{n+1}-\rho_{j}^{n+1/2}\right|+\left|\rho_{j}^{n+1/2}-\rho_{j}^{n}\right|.
\end{eqnarray}
We then estimate separately each term on the right hand side of the inequality \eqref{rhom+1_rhom}.\\
By the definition of the relaxation step \eqref{Source_disc}, for the first term on right hand side of \eqref{rhom+1_rhom} we have
\begin{eqnarray}
\left|\rho_{j}^{n+1}-\rho_{j}^{n+1/2}\right|&\leq&\Delta t\left|S_{\mathrm{on},j}^{n+1/2}-S_{\mathrm{off},j}^{n+1/2}\right|\nonumber\\
&\leq&\Delta t\mathbf{1}_{\mathrm{on},j} q_{\mathrm{on}}^{n+1/2}\left(1-\rho_{j}^{n+1/2}\right)\left(1-R_{\mathrm{on},j}^{n+1/2}\right)+\Delta t\mathbf{1}_{\mathrm{off},j}q_{\mathrm{off}}^{n+1/2}\rho_{j}^{n+1/2}\nonumber\\
\label{dif_rhon+1_rhon1/2}&\leq&\Delta t\|q_{\mathrm{on}}\|_{\L{\infty}([0,T])}\left(\mathbf{1}_{\mathrm{on},j}+\mathbf{1}_{\mathrm{on},j}\left|\rho_{j}^{n+1/2}\right|\right)+\Delta t\mathbf{1}_{\mathrm{off},j}\left\|q_{\mathrm{off}}\right\|_{\L{\infty}([0,T])}\left|\rho_{j}^{n+1/2}\right|,
\end{eqnarray}
then multiplying by $\Delta x$ and summing over all $j\in\Z$,
\begin{eqnarray}
\Delta x\sum_{j\in\Z}\left|\rho_{j}^{n+1}-\rho_{j}^{n+1/2}\right|&\leq&\Delta t\left\|q_{\mathrm{on}}\right\|_{\L{\infty}([0,T])}\left(\Delta x\sum_{j\in\Omega_{\mathrm{on}}^{k}}\mathbf{1}_{\mathrm{on},j}+\Delta x\sum_{j\in\Omega_{\mathrm{on}}^{k}}\mathbf{1}_{\mathrm{on},j}\left|\rho_{j}^{n+1/2}\right|\right)\nonumber\\
&&+\Delta t\left\|q_{\mathrm{off}}\right\|_{\L{\infty}([0,T])}\Delta x\sum_{j\in\Omega_{\mathrm{off}}^{k}}\mathbf{1}_{\mathrm{off},j}\left|\rho_{j}^{n+1/2}\right|\nonumber\\
&\leq&\Delta t\left\|q_{\mathrm{on}}\right\|_{\L{\infty}([0,T])}\left(1+\frac{\left\|\rho^{n+1/2}\right\|_{\L{1}(\R)}}{L}\right)\nonumber\\
&&+\Delta t\left\|q_{\mathrm{off}}\right\|_{\L{\infty}([0,T])}\frac{\|\rho^{n+1/2}\|_{\L{1}(\R)}}{L}\nonumber\\
&=&\Delta t\left\|q_{\mathrm{on}}\right\|_{\L{\infty}([0,T])}\left(1+\frac{\left\|\rho^{n}\right\|_{\L{1}(\R)}}{L}\right)+\Delta t\left\|q_{\mathrm{off}}\right\|_{\L{\infty}([0,T])}\frac{\|\rho^{n}\|_{\L{1}(\R)}}{L}\nonumber\\
\label{first_BVST_RHS}&=&\Delta t\left(\frac{\left\|q_{\mathrm{on}}\right\|_{\L{\infty}([0,T])}+\left\|q_{\mathrm{off}}\right\|_{\L{\infty}([0,T])}}{L}\right)\|\rho^{n}\|_{\L{1}(\R)}+\Delta t\frac{\left\|q_{\mathrm{on}}\right\|_{\L{\infty}([0,T])}}{L}
\end{eqnarray}

Now we analyze the second term of the right hand side \eqref{rhom+1_rhom}. Since the numerical flux defined in \eqref{rhon1/2} is Lipschitz continuous in both arguments with Lipschitz constant $\mathcal{L}$, defined by \eqref{L}, 
we obtain
\begin{eqnarray*}
\left|\rho_{j}^{n+1/2}-\rho_{j}^{n}\right|&=&\lambda\left|F_{j+1/2}(\rho_{j}^{n},R_{j+1/2}^{n})-F_{j-1/2}(\rho_{j-1}^{n},R_{j-1/2}^{n})\right|\nonumber\\
&\leq&\lambda\mathcal{L}\left(\left|\rho_{j}^{n}-\rho_{j-1}^{n}\right|+\left|R_{j+1/2}^{n}-R_{j-1/2}^{n}\right|\right),
\end{eqnarray*}
multiplying by $\Delta x$, summing over all $j\in\Z$ and by the Remark \ref{remark:bound_sum_Ron} we get
\begin{eqnarray}
\Delta x\sum_{j\in\Z}\left|\rho_{j}^{n+1/2}-\rho_{j}^{n}\right|&\leq&2\mathcal{L}\Delta t\sum_{j\in\Z}\left|\rho_{j}^{n}-\rho_{j-1}^{n}\right|\nonumber\\
\label{second_BVST_RHS}&=&2\mathcal{L}\Delta t\sum_{j\in\Z}\left|\rho_{j+1}^{n}-\rho_{j}^{n}\right|.
\end{eqnarray}
Collecting together \eqref{first_BVST_RHS} and \eqref{second_BVST_RHS}, and by using Lemma \ref{Lemma:L1_norm} and Proposition \ref{prop:BV_Space}  we have,
\begin{eqnarray}
\Delta x\sum_{j\in\Z}\left|\rho_{j}^{n+1}-\rho_{j}^{n}\right|&\leq&\Delta t\left(\frac{\left\|q_{\mathrm{on}}\right\|_{\L{\infty}([0,T])}+\left\|q_{\mathrm{off}}\right\|_{\L{\infty}([0,T])}}{L}\right)\|\rho^{n}\|_{\L{1}(\R)}+\Delta t\frac{\left\|q_{\mathrm{on}}\right\|_{\L{\infty}([0,T])}}{L}\nonumber\\
&&+2\mathcal{L}\Delta t\sum_{j\in\Z}\left|\rho_{j+1}^{n}-\rho_{j}^{n}\right|\nonumber\\
&\leq&\Delta t\left(\frac{\left\|q_{\mathrm{on}}\right\|_{\L{\infty}([0,T])}+\left\|q_{\mathrm{off}}\right\|_{\L{\infty}([0,T])}}{L}\right)\mathcal{C}_{1}(T)+\Delta t\frac{\left\|q_{\mathrm{on}}\right\|_{\L{\infty}([0,T])}}{L}\nonumber\\
&&\label{BV_Space_time2}+2\mathcal{L}\Delta te^{T\mathcal{H}}\left(TV(\rho_{0})+T\left(\frac{\left\|q_{\mathrm{on}}\right\|_{\L{\infty}([0,T])}+\left\|q_{\mathrm{off}}\right\|_{\L{\infty}([0,T])}}{L}\right)\right).
\end{eqnarray}
Then, collecting together \eqref{BV_Space_time1} and \eqref{BV_Space_time2} we get
\begin{eqnarray*}
&&\sum_{n=0}^{N_{T}-1}\sum_{j\in\Z}\Delta t\left|\rho_{j+1}^{n}-\rho_{j}^{n}\right|+\left(T-N_{T}\Delta t\right)\sum_{j\in\Z}\left|\rho_{j+1}^{N_{T}}-\rho_{j}^{N_{T}}\right|+\sum_{n=0}^{N_{T}-1}\sum_{j\in\Z}\Delta x\left|\rho_{j}^{n+1}-\rho_{j}^{n}\right|\\
&\leq&Te^{T\mathcal{H}}\left(\left(1+2\mathcal{L}\right)\left(TV(\rho_{0})+T\left(\frac{\left\|q_{\mathrm{on}}\right\|_{\L{\infty}([0,T])}+\left\|q_{\mathrm{off}}\right\|_{\L{\infty}([0,T])}}{L}\right)\right)\right)\\
&&+T\left(\frac{\left\|q_{\mathrm{on}}\right\|_{\L{\infty}([0,T])}+\left\|q_{\mathrm{off}}\right\|_{\L{\infty}([0,T])}}{L}\right)\mathcal{C}_{1}(T)+T\frac{\left\|q_{\mathrm{on}}\right\|_{\L{\infty}([0,T])}}{L}.
\end{eqnarray*}
\end{proof}
\subsection{Discrete Entropy Inequality}\label{subsec:Disc_Entropy_inequality}
\

We define, for $\kappa\in[0,1]$, 
\begin{eqnarray*}
G_{j+1/2}(u\vee \kappa)=uv(R_{j+1/2}), \quad \mathscr{F}_{j+1/2}^{\kappa}(u)=G_{j+1/2}(u\vee \kappa)-G_{j+1/2}(u\wedge \kappa),
\end{eqnarray*}
with $a\vee b=\max\{a,b\}$, and $a\wedge b=\min\{a,b\}$.
\begin{lemma}\label{Lemma:Entropy_ineq} Let $\rho_{0}\in\left(\L{1}\cap\BV\right)\left(\R;[0,1]\right)$. Assume that hypotheses \eqref{Assumption1} and CFL condition \eqref{CFL} hold. Then, the approximate solution $\rho_{\Delta}$ constructed by \textbf{Algorithm \ref{Algorithm_Scheme}} satisfies the following discrete entropy inequality: for $j\in\Z$, for $n=0,\ldots,N_{T}-1$ and for any $\kappa\in[0,1]$,
\begin{eqnarray*}
&&\left|\rho_{j}^{n+1}-\kappa\right|-\left|\rho_{j}^{n}-\kappa\right|+\lambda\left(\mathscr{F}_{j+1/2}^{k}\left(\rho_{j}^{n}\right)-\mathscr{F}_{j+1/2}^{k}\left(\rho_{j-1}^{n}\right)\right)\\
&&-\Delta t\sgn\left(\rho_{j}^{n+1}-\kappa\right)\left(S_{\mathrm{on}}\left(t^{n+1/2},x_{j},\rho_{j}^{n+1/2},R_{\mathrm{on},j}^{n+1/2}\right)-S_{\mathrm{off}}\left(t^{n+1/2},x_{j},\rho_{j}^{n+1/2}\right)\right)\\
&&+\lambda\sgn\left(\rho_{j}^{n+1}-\kappa\right)\kappa\left(v\left(R_{j+1/2}^{n}\right)-v\left(R_{j-1/2}^{n}\right)\right)\leq0.
\end{eqnarray*}
\end{lemma}
\begin{proof}
We set 
\begin{eqnarray*}
\mathscr{G}_{j}(u,w)&=&w-\lambda\left(G_{j+1/2}(w)-G_{j-1/2}(u)\right)\\
&=&w-\lambda\left(wv(R_{j+1/2})-uv(R_{j-1/2})\right). 
\end{eqnarray*}
Clearly $\rho_{j}^{n+1/2}=\mathscr{G}_{j}(\rho_{j-1}^{n},\rho_{j}^{n})$.\\
The map $\mathscr{G}_{j}$ is a monotone non-decreasing function with respect to each variable under the CFL condition \eqref{CFL} since we have
\begin{eqnarray*}
\frac{\partial\mathscr{G}}{\partial w}=1-\lambda v(R_{j+1/2})\geq0, \qquad \frac{\partial\mathscr{G}}{\partial u}=\lambda v(R_{j-1/2}).
\end{eqnarray*}
Moreover, we have the following identity
\begin{eqnarray*}
\mathscr{G}_{j}(\rho_{j-1}^{n}\vee\kappa,\rho_{j}^{n}\vee\kappa)-\mathscr{G}_{j}(\rho_{j-1}^{n}\vee\kappa,\rho_{j}^{n}\wedge\kappa)=\left|\rho_{j}^{n}-\kappa\right|-\lambda\left(\mathscr{F}_{j+1/2}^{k}\left(\rho_{j}^{n}\right)-\mathscr{F}_{j-1/2}^{k}\left(\rho_{j-1}^{n}\right)\right).
\end{eqnarray*}
Then, by monotonicity, the definition of scheme \eqref{rhon1/2} and by using $\left|a+b\right|\geq|a|+\sgn(a)b$, we get
\begin{eqnarray*}
&&\mathscr{G}_{j}(\rho_{j-1}^{n}\vee\kappa,\rho_{j}^{n}\vee\kappa)-\mathscr{G}_{j}(\rho_{j-1}^{n}\vee\kappa,\rho_{j}^{n}\wedge\kappa)\\
&\geq&\mathscr{G}_{j}(\rho_{j-1}^{n},\rho_{j}^{n})\vee\mathscr{G}_{j}(\kappa,\kappa)-\mathscr{G}_{j}(\rho_{j-1}^{n},\rho_{j}^{n})\wedge\mathscr{G}_{j}(\kappa,\kappa)\\
&=&\left|\mathscr{G}_{j}(\rho_{j-1}^{n},\rho_{j}^{n})-\mathscr{G}_{j}(\kappa,\kappa)\right|\\
&=&\left|\rho_{j}^{n+1/2}-\mathscr{G}_{j}(\kappa,\kappa)\right|\\
&=&\bigg|\rho_{j}^{n+1}-\kappa+\lambda\kappa\left(v(R_{j+1/2}^{n})-v(R_{j-1/2}^{n})\right)\\
&&-\Delta t\left(S_{\mathrm{on}}\left(t^{n+1/2},x_{j},\rho_{j}^{n+1/2},R_{\mathrm{on},j}^{n+1/2}\right)-S_{\mathrm{off}}\left(t^{n+1/2},x_{j},\rho_{j}^{n+1/2}\right)\right)\bigg|\\
&\geq&\left|\rho_{j}^{n+1}-\kappa\right|+\lambda\sgn\left(\rho_{j}^{n+1}-\kappa\right)\kappa\left(v(R_{j+1/2}^{n})-v(R_{j-1/2}^{n})\right)\\
&&-\Delta t\sgn\left(\rho_{j}^{n+1}-\kappa\right)\left(S_{\mathrm{on}}\left(t^{n+1/2},x_{j},\rho_{j}^{n+1/2},R_{\mathrm{on},j}^{n+1/2}\right)-S_{\mathrm{off}}\left(t^{n+1/2},x_{j},\rho_{j}^{n+1/2}\right)\right).
\end{eqnarray*}
\end{proof}
The following Theorem states the $\L{1}$-Lipschitz continuous dependence of solution to \eqref{nonRTM} on both the initial datum and the $q_{\mathrm{on}}$ and $q_{\mathrm{off}}$ functions. 
\begin{theorem}[Uniqueness]\label{uniqueness}
Let $\rho$ and $\tilde{\rho}$ be two solutions to problem \eqref{nonRTM} in the sense of Definition \ref{entropy_WS}, with initial data $\rho_0,\ \tilde{\rho}_{0}\in\L{1}\cap\BV\left(\R;[0,1]\right)$ respectively. Assume $v\in\C{2}\left([0,1],\R\right)$. Then, for a.e. $t\in[0,T]$,
\begin{eqnarray*}
\left\|\rho(t)-\tilde{\rho}(t)\right\|_{\L{1}(\R)}&\leq&e^{\mathcal{C}T}\left(\left\|\rho_{0}-\tilde{\rho}_{0}\right\|_{\L{1}(\R)}+\left\|q_{\mathrm{on}}-\tilde{q}_{\mathrm{on}}\right\|_{\L{1}([0,t])}+\left\|q_{\mathrm{off}}-\tilde{q}_{\mathrm{off}}\right\|_{\L{1}([0,T])}\right).
\end{eqnarray*}
\end{theorem}
\begin{proof}
The proof follows closely Theorem 5.6 of \cite{friedrich2020nonlocal}.\\

By using Kru\v{z}kov's doubling of variables technique we get 
\begin{eqnarray*}
\left\|\rho(T,\cdot)-\tilde{\rho}(T,\cdot)\right\|_{\L{1}(\R)}&\leq&\left\|\rho_{0}-\tilde{\rho}_{0}\right\|_{\L{1}(\R)}+\int_{0}^{T}\int_{\Omega_{\mathrm{on}}}\left|\tilde{\mathcal{S}}_{\mathrm{on}}\right|\dd x\dd t+\int_{0}^{T}\int_{\Omega_{\mathrm{off}}}\left|\tilde{\mathcal{S}}_{\mathrm{off}}\right|\dd x\dd t\nonumber\\
&&+\int_{0}^{T}\int_{\R}\left|\mathcal{V}\right|\left|\partial_{x}\rho(t,x)\right|\dd x \dd t+\int_{0}^{T}\int_{\R}\left|\mathcal{V}_{x}\right|\left|\rho(t,x)\right|\dd x\dd t,
\end{eqnarray*}
where 
\begin{eqnarray*}
\tilde{\mathcal{S}}_{\mathrm{on}}&=&S_{\mathrm{on}}\left(t,x,q_{\mathrm{on}},\rho,R_{\mathrm{on}}\right)-S_{\mathrm{on}}\left(t,x,\tilde{q}_{\mathrm{on}},\tilde{\rho},\tilde{R}_{\mathrm{on}}\right),\\
\tilde{\mathcal{S}}_{\mathrm{off}}&=&S_{\mathrm{off}}\left(t,x,q_{\mathrm{on}},\rho\right)-S_{\mathrm{off}}\left(t,x,\tilde{q}_{\mathrm{on}},\tilde{\rho}\right),\\
\mathcal{V}&=&v(R)-v(P),\\
\mathcal{V}_{x}&=&\partial_{x}v(R)-\partial_{x}v(P)
\end{eqnarray*}

Let us now estimate all the terms appearing in the right hand side of the above inequality.
We start bounding $\tilde{\mathcal{S}}_{\mathrm{on}}$ and $\tilde{\mathcal{S}}_{\mathrm{off}}$ terms:

\begin{eqnarray*}
\int_{0}^{T}\int_{\Omega_{\mathrm{on}}}\left|\tilde{\mathcal{S}}_{\mathrm{on}}\right|\dd x\dd t&=&\int_{0}^{T}\int_{\Omega_{\mathrm{on}}}\left|S_{\mathrm{on}}\left(t,x,q_{\mathrm{on}},\rho,R_{\mathrm{on}}\right)-S_{\mathrm{on}}\left(t,x,\tilde{q}_{\mathrm{on}},\tilde{\rho},\tilde{R}_{\mathrm{on}}\right)\right|\dd x\dd t\\
&\leq&\int_{0}^{T}\int_{\Omega_{\mathrm{on}}}\left(\left|\tilde{\mathcal{S}}_{\mathrm{on}}^{1}\right|+\left|\tilde{\mathcal{S}}_{\mathrm{on}}^{2}\right|+\left|\tilde{\mathcal{S}}_{\mathrm{on}}^{3}\right|\right)\dd x\dd t,
\end{eqnarray*}
where 
\begin{eqnarray*}
\tilde{\mathcal{S}}_{\mathrm{on}}^{1}&=&S_{\mathrm{on}}\left(t,x,q_{\mathrm{on}},\rho,R_{\mathrm{on}}\right)-S_{\mathrm{on}}\left(t,x,q_{\mathrm{on}},\rho,\tilde{R}_{\mathrm{on}}\right),\\
\tilde{\mathcal{S}}_{\mathrm{on}}^{2}&=&S_{\mathrm{on}}\left(t,x,q_{\mathrm{on}},\rho,\tilde{R}_{\mathrm{on}}\right)-S_{\mathrm{on}}\left(t,x,q_{\mathrm{on}},\tilde{\rho},\tilde{R}_{\mathrm{on}}\right),\\ \tilde{\mathcal{S}}_{\mathrm{on}}^{3}&=&S_{\mathrm{on}}\left(t,x,q_{\mathrm{on}},\tilde{\rho},\tilde{R}_{\mathrm{on}}\right)-S_{\mathrm{on}}\left(t,x,\tilde{q}_{\mathrm{on}},\tilde{\rho},\tilde{R}_{\mathrm{on}}\right).
\end{eqnarray*}
First we going to bound $\tilde{\mathcal{S}}_{\mathrm{on}}^{1}$ term ,
\begin{eqnarray*}
\left|\tilde{\mathcal{S}}_{\mathrm{on}}^{1}\right|&=&\left|\mathbf{1}_{\mathrm{on}}q_{\mathrm{on}}\left(1-\rho\right)\left(\left(1-R_{\mathrm{on}}\right)-\left(1-\tilde{R}_{\mathrm{on}}\right)\right)\right|\\
&\leq&\frac{\left\|q_{\mathrm{on}}\right\|_{\L{\infty}([0,T])}}{L}\left|\tilde{R}_{\mathrm{on}}-R_{\mathrm{on}}\right|,
\end{eqnarray*}
thus
\begin{eqnarray*}
\int_{0}^{T}\int_{\Omega_{\mathrm{on}}}\left|\tilde{\mathcal{S}}_{\mathrm{on}}^{1}\right|\dd x\dd t&\leq&\frac{\left\|q_{\mathrm{on}}\right\|_{\L{\infty}([0,T])}}{L}\int_{0}^{T}\int_{\Omega_{\mathrm{on}}}\left|\tilde{R}_{\mathrm{on}}-R_{\mathrm{on}}\right|\dd x\dd t\\
&\leq&\frac{\left\|q_{\mathrm{on}}\right\|_{\L{\infty}([0,T])}}{L}\int_{0}^{T}\left\|\tilde{R}_{\mathrm{on}}-R_{\mathrm{on}}\right\|_{\L{1}(\Omega_{\mathrm{on}})}.
\end{eqnarray*}
Observe that 
$$\left\|R_{\mathrm{on}}-\tilde{R}_{\mathrm{on}}\right\|_{\L{1}(\Omega_{\mathrm{on}})}\leq\left\|\rho(t,\cdot)-\tilde{\rho}(t,\cdot)\right\|_{\L{1}(\Omega_{\mathrm{on}})},$$
since $\int_{\R}\omega_{\eta}(x)\dd x=1$. Then,\\
\begin{eqnarray*}
\int_{0}^{T}\int_{\Omega_{\mathrm{on}}}\left|\tilde{\mathcal{S}}_{\mathrm{on}}^{1}\right|\dd x\dd t&\leq&\frac{\left\|q_{\mathrm{on}}\right\|_{\L{\infty}([0,T])}}{L}\int_{0}^{T}\left\|\rho(t,\cdot)-\tilde{\rho}(t,\cdot)\right\|_{\L{1}(\Omega_{\mathrm{on}})}\dd t\\
&\leq&\frac{\left\|q_{\mathrm{on}}\right\|_{\L{\infty}([0,T])}}{L}\int_{0}^{T}\left\|\rho(t,\cdot)-\tilde{\rho}(t,\cdot)\right\|_{\L{1}(\R)}\dd t.
\end{eqnarray*}
Now we going to bound $\tilde{\mathcal{S}}_{\mathrm{on}}^{2}$.
\begin{eqnarray*}
\left|\tilde{\mathcal{S}}_{\mathrm{on}}^{2}\right|&=&\left|\mathbf{1}_{\mathrm{on}}q_{\mathrm{on}}\left(1-\tilde{R}_{\mathrm{on}}\right)\left(1-\rho\right)\left(\tilde{\rho}-\rho\right)\right|\\
&\leq&\frac{\left\|q_{\mathrm{on}}\right\|_{\L{\infty}([0,T])}}{L}\left|\rho-\tilde{\rho}\right|.
\end{eqnarray*}
Integrating in time and space we have 
\begin{eqnarray*}
\int_{0}^{T}\int_{\Omega_{\mathrm{on}}}\left|\tilde{\mathcal{S}}_{\mathrm{on}}^{2}\right|\dd x\dd t&\leq&\frac{\left\|q_{\mathrm{on}}\right\|_{\L{\infty}([0,T])}}{L}\int_{0}^{T}\left\|\rho(t,\cdot)-\tilde{\rho}(t,\cdot)\right\|_{\L{1}(\Omega_{\mathrm{on}})}\dd t\\
&\leq&\frac{\left\|q_{\mathrm{on}}\right\|_{\L{\infty}([0,T])}}{L}\int_{0}^{T}\left\|\rho(t,\cdot)-\tilde{\rho}(t,\cdot)\right\|_{\L{1}(\R)}\dd t.
\end{eqnarray*}
 Bounding $\tilde{\mathcal{S}}_{\mathrm{on}}^{3}$,
\begin{eqnarray*}
\left|\tilde{\mathcal{S}}_{\mathrm{on}}^{3}\right|&=&\left|\mathbf{1}_{\mathrm{on}}\left(1-\tilde{\rho}\right)\left(1-\tilde{R}_{\mathrm{on}}\right)\left(q_{\mathrm{on}}-\tilde{q}_{\mathrm{on}}\right)\right|\\
&\leq&\frac{\left|q_{\mathrm{on}}-\tilde{q}_{\mathrm{on}}\right|}{L},
\end{eqnarray*}
thus 
\begin{eqnarray*}
\int_{0}^{T}\int_{\Omega_{\mathrm{on}}}\left|\tilde{\mathcal{S}}_{\mathrm{on}}^{3}\right|\dd x\dd t&\leq&\frac{1}{L}\int_{0}^{T}\int_{\Omega_{\mathrm{on}}}\left|q_{\mathrm{on}}-\tilde{q}_{\mathrm{on}}\right| \dd x\dd t\\
&\leq&\left\|q_{\mathrm{on}}-\tilde{q}_{\mathrm{on}}\right\|_{\L{1}([0,T])}.
\end{eqnarray*}
Therefore, we get the following estimate
\begin{eqnarray}
&&\int_{0}^{T}\int_{\Omega_{\mathrm{on}}}\left|\tilde{\mathcal{S}}_{\mathrm{on}}\right|\dd x\dd t\nonumber\\
&&\label{Son_bound}\leq2\|q_{\mathrm{on}}\|_{\L{\infty}([0,T])}\int_{0}^{T}\left\|\rho(t,\cdot)-\tilde{\rho}(t,\cdot)\right\|_{\L{1}(\R)}\dd t+\left\|q_{\mathrm{on}}-\tilde{q}_{\mathrm{on}}\right\|_{\L{1}([0,T])}.
\end{eqnarray}
Regarding $\tilde{\mathcal{S}}_{\mathrm{off}}$ term, we proceed in a similar way like above and we get
\begin{eqnarray*}
\left|\tilde{\mathcal{S}}_{\mathrm{off}}\right|&=&\left|\mathbf{1}_{\mathrm{off}}q_{\mathrm{off}}\rho-\mathbf{1}_{\mathrm{off}}\tilde{q}_{\mathrm{off}}\tilde{\rho}\right|\\
&\leq&\left|\tilde{\mathcal{S}}_{\mathrm{off}}^{1}\right|+\left|\tilde{\mathcal{S}}_{\mathrm{off}}^{2}\right|,
\end{eqnarray*}
where
\begin{eqnarray*}
\tilde{\mathcal{S}}_{\mathrm{off}}^{1}&=&S_{\mathrm{off}}\left(t,x,q_{\mathrm{off}},\rho\right)-S_{\mathrm{off}}\left(t,x,q_{\mathrm{off}},\tilde{\rho}\right),\\
\tilde{\mathcal{S}}_{\mathrm{off}}^{2}&=&S_{\mathrm{off}}\left(t,x,q_{\mathrm{off}},\tilde{\rho}\right)-S_{\mathrm{off}}\left(t,x,\tilde{q}_{\mathrm{off}},\tilde{\rho}\right).
\end{eqnarray*}
Then,
\begin{eqnarray*}
\int_{0}^{T}\int_{\Omega_{\mathrm{off}}}\left|\tilde{\mathcal{S}}_{\mathrm{off}}^{1}\right|\dd x\dd t&\leq& 
\frac{\left\|{q}_{\mathrm{off}}\right\|_{\L{\infty}([0,T])}}{L}\int_{0}^{T}\left\|\rho(t,\cdot)-\tilde{\rho}(t,\cdot)\right\|_{\L{1}(\Omega_{\mathrm{off}})}\dd t\\
&\leq&\frac{\left\|q_{\mathrm{off}}\right\|_{\L{\infty}([0,T])}}{L}\int_{0}^{T}\left\|\rho(t,\cdot)-\tilde{\rho}(t,\cdot)\right\|_{\L{1}(\R)}\dd t,
\end{eqnarray*}
and 
\begin{eqnarray*}
\int_{0}^{T}\int_{\Omega_{\mathrm{off}}}\left|\tilde{\mathcal{S}}_{\mathrm{off}}^{2}\right|\dd x\dd t\leq\left\|q_{\mathrm{off}}-\tilde{q}_{\mathrm{off}}\right\|_{\L{1}([0,T])}.
\end{eqnarray*}
Thus, we get
\begin{eqnarray}
&&\int_{0}^{T}\int_{\Omega_{\mathrm{off}}}\left|\mathcal{S}_{\mathrm{off}}\right|\dd x\dd t\nonumber\\
&&\label{Soff_bound}\leq\frac{\left\|{q}_{\mathrm{off}}\right\|_{\L{\infty}([0,T])}}{L}\int_{0}^{T}\left\|\rho(t,\cdot)-\tilde{\rho}(t,\cdot)\right\|_{\L{1}(\R)}\dd t+\left\|q_{\mathrm{off}}-\tilde{q}_{\mathrm{off}}\right\|_{\L{1}([0,T])}.
\end{eqnarray}
\normalsize{Next, focus on $\mathcal{V}$, by using the following estimate
\begin{eqnarray*}
\left|\mathcal{V}\right|\leq\omega_{\eta}(0)\left\|v'\right\|_{\L{\infty}([0,1])}\left\|\rho(t,\cdot)-\tilde{\rho}(t,\cdot)\right\|_{\L{1}(\R)},
\end{eqnarray*}
we obtain}

\begin{eqnarray}
&&\int_{0}^{T}\int_{\R}\left|\mathcal{V}\right|\left|\partial_{x}\rho(t,x)\right|\dd x \dd t\nonumber\\
&&\leq\label{diff_vel2}\omega_{\eta}(0)\left\|v'\right\|_{\L{\infty}([0,1])}\sup_{t\in[0,T]}\left\|\rho(t,\cdot)\right\|_{\mathbf{TV}(\R)}\int_{0}^{T}\left\|\rho(t,\cdot)-\tilde{\rho}(t,\cdot)\right\|_{\L{1}(\R)}\dd t.
\end{eqnarray}

Next, we pass to $\mathcal{V}_{x}$. Following \cite{friedrich2020nonlocal} we compute

\begin{eqnarray*}
\left|\mathcal{V}_{x} \right|&\leq&\left(2\left(\omega_{\eta}(0)\right)^{2}\left\|v''\right\|_{\L{\infty}([0,1])}+\left\|v'\right\|_{\L{\infty}([0,1])}\left\|\omega_{\eta}'\right\|_{\L{\infty}([0,\eta])}\right)\left\|\rho(t,\cdot)-\tilde{\rho}(t,\cdot)\right\|_{\L{1}(\R)}\\
&&+\omega_{\eta}(0)\|v'\|_{\L{\infty}([0,1])}\left(\left|\rho-\tilde{\rho}\right|\left(t,x+\eta\right)+\left|\rho-\tilde{\rho}\right|\left(t,x\right)\right),
\end{eqnarray*}
thus
\begin{eqnarray}
\int_{0}^{T}\int_{\R}\left|\mathcal{V}_{x} \right|\left|\rho(t,x)\right|\dd x\dd t
\label{diff_Vx2}&\leq&\mathcal{W}\int_{0}^{T}\left\|\rho(t,\cdot)-\tilde{\rho}(t,\cdot)\right\|_{\L{1}(\R)}\dd t,
\end{eqnarray}
where $$\mathcal{W}=\left(2\left(\omega_{\eta}(0)\right)^{2}\left\|v''\right\|_{\L{\infty}([0,1])}+\left\|v'\right\|_{\L{\infty}([0,1])}\left\|\omega_{\eta}'\right\|_{\L{\infty}([0,\eta])}\right)\mathcal{C}_{1}(t)+2\omega_{\eta}(0)\left\|v'\right\|_{\L{\infty}([0,1])}.$$

Collecting together \eqref{Son_bound}, \eqref{Soff_bound}, \eqref{diff_vel2} and \eqref{diff_Vx2} we get
\begin{eqnarray}\label{L1contractive}
\left\|\rho(T,\cdot)-\tilde{\rho}(T,\cdot)\right\|_{\L{1}(\R)}&\leq&\left\|\rho_{0}-\tilde{\rho}_{0}\right\|_{\L{1}(\R)}+\left(\left\|q_{\mathrm{on}}-\tilde{q}_{\mathrm{on}}\right\|_{\L{1}([0,t])}+\left\|q_{\mathrm{off}}-\tilde{q}_{\mathrm{off}}\right\|_{\L{1}([0,t])}\right)\nonumber\\
&&+\mathcal{C}\int_{0}^{T}\left\|\rho(t,\cdot)-\tilde{\rho}(t,\cdot)\right\|_{\L{1}(\R)}\dd t,
\end{eqnarray}
where 
\begin{eqnarray}\label{C}
\mathcal{C}&=&2\left\|q_{\mathrm{on}}\right\|_{\L{\infty}([0,T])}+\left\|{q}_{\mathrm{off}}\right\|_{\L{\infty}([0,T])}+\omega_{\eta}(0)\left\|v'\right\|_{\L{\infty}([0,1])}\sup_{t\in[0,T]}\left\|\rho(t,\cdot)\right\|_{\mathbf{TV}(\R)}+\mathcal{W}
\end{eqnarray}
An application of Gronwall Lemma to \eqref{L1contractive} completes the proof.
\end{proof}
\subsection{Proof of theorem \ref{teo:uniq_exist}}
The convergence of the approximate solutions constructed by \textbf{Algorithm \ref{Algorithm_Scheme}} towards the unique weak entropy solution can be proven by applying Helly's compactness theorem. The latter can be applied due to Lemma \ref{Lemma:Maximum_pple} and Proposition \ref{prop:BV_SpaceTime} and states that there exists a sub-sequence of approximate solution $\rho_{\Delta}$ that converges in $\L{1}$ to a function $\rho\in\L{\infty}\left([0,T]\times\R;[0,1]\right)$. Following a Lax-Wendroff type argument, we can show that the limit function $\rho$ is a weak entropy solution of \eqref{nonRTM} in the sense of Definition \ref{entropy_WS}. Together with the uniqueness result in Theorem \ref{uniqueness}. this concludes the proof of Theorem \ref{teo:uniq_exist}. \\ 

\subsection{Existence for Model 2}
In this section we consider the problem \eqref{nonRTM} with the $S_{\mathrm{on}}$ \eqref{nolreact2}. In \textbf{Algorithm \ref{Algorithm_Scheme}} we substitute $S_{\mathrm{on}}$ term in the reaction step \eqref{Source_disc} by \eqref{Son_def2},
thus now the term \eqref{Source_disc} is given by
\begin{eqnarray}\label{rhon+1_Son2}
\rho_{j}^{n+1}=\rho_{j}^{n+1/2}+\Delta t\mathbf{1}_{\mathrm{on},j}q_{\mathrm{on}}^{n+1/2}\left(1-\max\left\{\rho_{j}^{n+1/2},R_{\mathrm{on},j}^{n+1/2}\right\}\right)-\Delta t\mathbf{1}_{\mathrm{off},j}q_{\mathrm{off}}^{n+1/2}\rho_{j}^{n+1/2}. 
\end{eqnarray}
\begin{lemma}[Maximum Principle]\label{Max_ppio_Son2} Let $\rho_{0}\in\L{\infty}(\R;[0,1])$. Let hypotheses \eqref{Assumption1} and CFL condition \eqref{CFL} hold,
then for all $t>0$ and $x\in\R$ the piece-wise constant approximate solution $\rho_{\Delta}$ constructed through \textbf{Algorithm \ref{Algorithm_Scheme}} is such that
\begin{eqnarray*}
0\leq \rho_{\Delta}(t,x)\leq1.    
\end{eqnarray*}
\end{lemma}
\begin{proof}
The proof is made by induction. We assume that $0\leq\rho_{j}^{n}\leq1$ for all $j\in\Z$. Consider the step \eqref{rhon1/2} of \textbf{Algorithm \ref{Algorithm_Scheme}}, by CFL condition \eqref{CFL} we have $0\leq\rho_{j}^{n+1/2}\leq 1$ for $j\in\Z$.\\

Now focus on the remaining step, involving the source term. 
\begin{eqnarray*}
\rho_{j}^{n+1}&=&\rho_{j}^{n+1/2}+\Delta t\mathbf{1}_{\mathrm{on},j}q_{\mathrm{on}}^{n+1/2}\left(1-\max\left\{\rho_{j}^{n+1/2},R_{\mathrm{on},j}^{n+1/2}\right\}\right)-\Delta t\mathbf{1}_{\mathrm{off},j}q_{\mathrm{off}}^{n+1/2}\rho_{j}^{n+1/2}\\
&=&\rho_{j}^{n+1/2}+\Delta t\mathbf{1}_{\mathrm{on},j}q_{\mathrm{on}}^{n+1/2}\left(1-\frac{\rho_{j}^{n+1/2}+R_{\mathrm{on},j}^{n+1/2}+\left|\rho_{j}^{n+1/2}-R_{\mathrm{on},j}^{n+1/2}\right|}{2}\right)\\
&&-\Delta t\mathbf{1}_{\mathrm{off},j}q_{\mathrm{off}}^{n+1/2}\rho_{j}^{n+1/2}\\
&=&\rho_{j}^{n+1/2}+\Delta t\mathbf{1}_{\mathrm{on},j}q_{\mathrm{on}}^{n+1/2}-\frac{\Delta t}{2}\mathbf{1}_{\mathrm{on},j}q_{\mathrm{on}}^{n+1/2}\rho_{j}^{n+1/2}-\frac{\Delta t}{2}\mathbf{1}_{\mathrm{on},j}q_{\mathrm{on}}^{n+1/2}R_{\mathrm{on},j}^{n+1/2}\\
&&-\frac{\Delta t}{2}\mathbf{1}_{\mathrm{on},j}q_{\mathrm{on}}^{n+1/2}\left|\rho_{j}^{n+1/2}-R_{\mathrm{on},j}^{n+1/2}\right|-\Delta t\mathbf{1}_{\mathrm{off},j}q_{\mathrm{off}}^{n+1/2}\rho_{j}^{n+1/2}\\
&\leq&\rho_{j}^{n+1/2}+\Delta t\mathbf{1}_{\mathrm{on},j}q_{\mathrm{on}}^{n+1/2}-\frac{\Delta t}{2}\mathbf{1}_{\mathrm{on},j}q_{\mathrm{on}}^{n+1/2}\rho_{j}^{n+1/2}-\cancel{\frac{\Delta t}{2}\mathbf{1}_{\mathrm{on},j}q_{\mathrm{on}}^{n+1/2}R_{\mathrm{on},j}^{n+1/2}}\\
&&+\cancel{\frac{\Delta t}{2}\mathbf{1}_{\mathrm{on},j}q_{\mathrm{on}}^{n+1/2}\left|R_{\mathrm{on},j}^{n+1/2}\right|}-\frac{\Delta t}{2}\mathbf{1}_{\mathrm{on},j}q_{\mathrm{on}}^{n+1/2}\left|\rho_{j}^{n+1/2}\right|-\Delta t\mathbf{1}_{\mathrm{off},j}q_{\mathrm{off}}^{n+1/2}\rho_{j}^{n+1/2}\\
&=&\rho_{j}^{n+1/2}+\Delta t\mathbf{1}_{\mathrm{on},j}q_{\mathrm{on}}^{n+1/2}-\Delta t\mathbf{1}_{\mathrm{on},j}q_{\mathrm{on}}^{n+1/2}\rho_{j}^{n+1/2}-\Delta t\mathbf{1}_{\mathrm{off},j}q_{\mathrm{off}}^{n+1/2}\rho_{j}^{n+1/2}\\
&=&\label{Max_Pple_Son2}\left(1-\Delta t\left(\mathbf{1}_{\mathrm{on},j}q_{\mathrm{on}}^{n+1/2}+\mathbf{1}_{\mathrm{off},j}q_{\mathrm{off}}^{n+1/2}\right)\right)\rho_{j}^{n+1/2}+\Delta t\mathbf{1}_{\mathrm{on},j}q_{\mathrm{on}}^{n+1/2},
\end{eqnarray*}
now we can proceed as in Lemma \ref{Lemma:Maximum_pple}.
\end{proof}
\begin{lemma}\label{L1_norm_Son2}
Let $\rho_{0}\in\L{1}(\R,[0,1])$. Let \eqref{Assumption1} and the CFL condition \eqref{CFL} hold. Then, the piece-wise constant approximate solution $\rho_{\Delta}$ constructed through \textbf{Algorithm \ref{Algorithm_Scheme}} satisfies,
\begin{eqnarray*}
\left\|\rho_{\Delta}(t)\right\|_{\L{1}(\R)}\leq\mathcal{C}_{1}(t),
\end{eqnarray*}
where $\mathcal{C}_{1}$ like in \eqref{C_1}.
\end{lemma}
\begin{proof}
By \eqref{Max_Pple_Son2} and CFL condition \eqref{CFL} we have
\begin{eqnarray*}
\left|\rho_{j}^{n+1}\right|\leq\left|\rho_{j}^{n+1/2}\right|+\Delta t\mathbf{1}_{\mathrm{on},j}q_{\mathrm{on}}^{n+1/2}\left(1-\left|\rho_{j}^{n+1/2}\right|\right)-\Delta t \mathbf{1}_{\mathrm{off},j}q_{\mathrm{off}}^{n+1/2}\left|\rho_{j}^{n+1/2}\right|,
\end{eqnarray*}
this cases reduce to \eqref{abs_rhon+1} and we can proceed as in Lemma \ref{Lemma:L1_norm}. 
\end{proof}
\subsection{BV estimates}\label{BV_estimates_Son2}
\begin{lemma}\label{Lemma:Lipschitz_Son2}
The map $S_{\mathrm{on}}$ given in \eqref{rhon+1_Son2} is Lipschitz continuous in second, third and fourth argument with Lipschitz constant $\left\|q_{\mathrm{on}}\right\|_{\L{\infty}([0,T])}$.
\end{lemma}
\begin{proof}
\begin{eqnarray*}
\left|S_{\mathrm{on}}(t,x,\rho,R_{\mathrm{on}})-S_{\mathrm{on}}(t,\tilde{x},\tilde{\rho},\tilde{R}_{\mathrm{on}})\right|&\leq&\mathcal{S}_{1}+\mathcal{S}_{2}+\mathcal{S}_{3},
\end{eqnarray*}
where
\begin{eqnarray*}
\mathcal{S}_{1}&=&\left|S_{\mathrm{on}}(t,x,\rho,R_{\mathrm{on}})-S_{\mathrm{on}}(t,x,\tilde{\rho},R_{\mathrm{on}})\right|\\
\mathcal{S}_{2}&=&\left|S_{\mathrm{on}}(t,x,\tilde{\rho},R_{\mathrm{on}})-S_{\mathrm{on}}(t,{x},\tilde{\rho},\tilde{R}_{\mathrm{on}})\right|\\
\mathcal{S}_{3}&=&\left|S_{\mathrm{on}}(t,x,\tilde{\rho},\tilde{R}_{\mathrm{on}})-S_{\mathrm{on}}(t,\tilde{x},\tilde{\rho},\tilde{R}_{\mathrm{on}})\right|.
\end{eqnarray*}
by the definition of $S_{\mathrm{on}}$ term we have 
\begin{eqnarray*}
\mathcal{S}_{1}&\leq&\left\|q_{\mathrm{on}}\right\|_{\L{\infty}([0,T])}\bigg|1-\max\left\{\rho,R_{\mathrm{on}}\right\}-\left(1-\max\left\{\tilde{\rho},R_{\mathrm{on}}\right\}\right)\bigg|\\
&=&\left\|q_{\mathrm{on}}\right\|_{\L{\infty}([0,T])}\bigg|\max\left\{\tilde{\rho},R_{\mathrm{on}}\right\}-\max\left\{\rho,R_{\mathrm{on}}\right\}\bigg|\\
&=&\frac{\left\|q_{\mathrm{on}}\right\|_{\L{\infty}([0,T])}}{2}\bigg|\tilde{\rho}+R_{\mathrm{on}}+\left|\tilde{\rho}-R_{\mathrm{on}}\right|-\left(\rho+R_{\mathrm{on}}+\left|\rho-R_{\mathrm{on}}\right|\right)\bigg|\\
&=&\frac{\left\|q_{\mathrm{on}}\right\|_{\L{\infty}([0,T])}}{2}\bigg|\tilde{\rho}-\rho+\left|\tilde{\rho}-R_{\mathrm{on}}\right|-\left|\rho-R_{\mathrm{on}}\right|\bigg|\\
&\leq&\frac{\left\|q_{\mathrm{on}}\right\|_{\L{\infty}([0,T])}}{2}\bigg|\tilde{\rho}-\rho+\left|\tilde{\rho}-\rho\right|+\cancel{\left|\rho-R_{\mathrm{on}}\right|}-\cancel{\left|\rho-R_{\mathrm{on}}\right|}\bigg|\\
&\leq&\frac{\left\|q_{\mathrm{on}}\right\|_{\L{\infty}([0,T])}}{2}\left(\left|\tilde{\rho}-\rho\right|+\left|\tilde{\rho}-\rho\right|\right)\\
&=&\left\|q_{\mathrm{on}}\right\|_{\L{\infty}([0,T])}\left|\tilde{\rho}-\rho\right|.
\end{eqnarray*}
Pass now to $\mathcal{S}_{2}$:
\begin{eqnarray*}
\mathcal{S}_{2}&\leq&\left\|q_{\mathrm{on}}\right\|_{\L{\infty}([0,T])}\bigg|\max\left\{\tilde{\rho},\tilde{R}_{\mathrm{on}}\right\}-\max\left\{\tilde{\rho},R_{\mathrm{on}}\right\}\bigg|\\
&=&\frac{\left\|q_{\mathrm{on}}\right\|_{\L{\infty}([0,T])}}{2}\bigg|\tilde{\rho}+\tilde{R_{\mathrm{on}}}+\left|\tilde{\rho}-\tilde{R}_{\mathrm{on}}\right|-\left(\tilde{\rho}+R_{\mathrm{on}}+\left|\tilde{\rho}-R_{\mathrm{on}}\right|\right)\bigg|\\
&=&\frac{\left\|q_{\mathrm{on}}\right\|_{\L{\infty}([0,T])}}{2}\bigg|\tilde{R}_{\mathrm{on}}-R_{\mathrm{on}}+\left|\tilde{\rho}-\tilde{R_{\mathrm{on}}}\right|-\left|\tilde{\rho}-R_{\mathrm{on}}\right|\bigg|\\
&=&\frac{\left\|q_{\mathrm{on}}\right\|_{\L{\infty}([0,T])}}{2}\bigg|\tilde{R}_{\mathrm{on}}-R_{\mathrm{on}}+\left|\tilde{\rho}-R_{\mathrm{on}}+R_{\mathrm{on}}-\tilde{R_{\mathrm{on}}}\right|-\left|\tilde{\rho}-R_{\mathrm{on}}\right|\bigg|\\
&\leq&\frac{\left\|q_{\mathrm{on}}\right\|_{\L{\infty}([0,T])}}{2}\bigg|\tilde{R}_{\mathrm{on}}-R_{\mathrm{on}}+\cancel{\left|\tilde{\rho}-R_{\mathrm{on}}\right|}+\left|\tilde{R}_{\mathrm{on}}-R_{\mathrm{on}}\right|-\cancel{\left|\tilde{\rho}-R_{\mathrm{on}}\right|}\bigg|\\
&\leq&\left\|q_{\mathrm{on}}\right\|_{\L{\infty}([0,T])}\left|R_{\mathrm{on}}-\tilde{R}_{\mathrm{on}}\right|.
\end{eqnarray*}
Next, we analyze the $\mathcal{S}_{3}$ term:
\begin{eqnarray*}
\mathcal{S}_{3}&=&\bigg|\mathbf{1}_{\mathrm{on}}q_{\mathrm{on}}\left(1-\max\left\{\tilde{\rho},\tilde{R}_{\mathrm{on}}\right\}\right)-\mathbf{\tilde{1}}_{\mathrm{on}}q_{\mathrm{on}}\left(1-\max\left\{\tilde{\rho},\tilde{R}_{\mathrm{on}}\right\}\right)\bigg|\\
&\leq&\left\|q_{\mathrm{on}}\right\|_{\L{\infty}([0,T])}\bigg|\mathbf{1}_{\mathrm{on}}-\mathbf{1}_{\mathrm{on}}\max\left\{\tilde{\rho},\tilde{R}_{\mathrm{on}}\right\}-\mathbf{\tilde{1}}_{\mathrm{on}}+\mathbf{\tilde{1}}_{\mathrm{on}}\max\left\{\tilde{\rho},\tilde{R}_{\mathrm{on}}\right\}\bigg|\\
&=&\left\|q_{\mathrm{on}}\right\|_{\L{\infty}([0,T])}\bigg|\mathbf{1}_{\mathrm{on}}-\mathbf{\tilde{1}}_{\mathrm{on}}-\frac{1}{2}\left(\tilde{\rho}+\tilde{R}_{\mathrm{on}}+\left|\tilde{\rho}-\tilde{R}_{\mathrm{on}}\right|\right)\left(\mathbf{1}_{\mathrm{on}}-\mathbf{\tilde{1}}_{\mathrm{on}}\right)\bigg|\\
&\leq&\left\|q_{\mathrm{on}}\right\|_{\L{\infty}([0,T])}\bigg|1-\frac{1}{2}\left(\tilde{\rho}+\cancel{\tilde{R}_{\mathrm{on}}}-\cancel{\left|\tilde{R}_{\mathrm{on}}\right|}+\left|\tilde{\rho}\right|\right)\bigg|\bigg|\mathbf{1}_{\mathrm{on}}-\mathbf{\tilde{1}}_{\mathrm{on}}\bigg|\\
&=&\left\|q_{\mathrm{on}}\right\|_{\L{\infty}([0,T])}\left|\mathbf{1}_{\mathrm{on}}-\mathbf{\tilde{1}}_{\mathrm{on}}\right|\left|1-\tilde{\rho}\right|\\
&\leq&\left\|q_{\mathrm{on}}\right\|_{\L{\infty}([0,T])}\left|\mathbf{1}_{\mathrm{on}}-\mathbf{\tilde{1}}_{\mathrm{on}}\right|.
\end{eqnarray*}
\end{proof}
\begin{proposition}[$\BV$ estimate in space]\label{prop:BV_Space_Son2} 
Let $\rho_{0}\in\left(\L{1}\cap\BV\right)\left(\R;[0,1]\right).$ Assume that the hypotheses \eqref{Assumption1} and CFL condition \eqref{CFL} hold. Then, for $n=0,\ldots,N_{T}-1$ the following estimate holds
\begin{eqnarray*}
\sum_{j\in\Z}\left|\rho_{j+1}^{n}-\rho_{j}^{n}\right|\leq e^{T\mathcal{H}}\left(TV(\rho^{0})+T\left(\frac{\left\|q_{\mathrm{on}}\right\|_{\L{\infty}([0,T])}+\left\|q_{\mathrm{off}}\right\|_{\L{\infty}([0,T])}}{L}\right)\right),
\end{eqnarray*}
with $\mathcal{H}$ like in \eqref{H}.
\end{proposition}
\begin{proof}
Due to the results obtained in Lemma \ref{Lemma:Lipschitz_Son2}, the proof is analogous to that one of Proposition \ref{prop:BV_Space}.
\end{proof}
\begin{proposition}[$\BV$ estimate in space and time]\label{prop:BV_SpaceTime_Son2}
Let hypotheses \eqref{Assumption1} hold, $\rho_{0}\in\left(\L{1}\cap\BV\right)\left(\R;[0,1]\right)$. If the CFL condition \eqref{CFL} holds, then, for every $T>0$ the following discrete space and time total variation estimate is satisfied:  
\begin{eqnarray*}
TV(\rho_{\Delta};[0,T]\times\R)
&\leq& T\mathcal{C}_{xt}(T),
\end{eqnarray*}
with 
$\mathcal{C}_{xt}(T)$ defined in \eqref{C_xt}.
\end{proposition}
\begin{proof}
For this proof we need to compute the following estimate,
\begin{eqnarray*}
\left|\rho_{j}^{n+1}-\rho_{j}^{n+1/2}\right|&\leq&\Delta t\left|S_{\mathrm{on},j}^{n+1/2}-S_{\mathrm{off},j}^{n+1/2}\right|\\
&\leq&\Delta t\mathbf{1}_{\mathrm{on},j}\left\|q_{\mathrm{on}}\right\|_{\L{\infty}([0,T])}\left|1-\frac{1}{2}\left(\rho_{j}^{n+1/2}+R_{\mathrm{on},j}^{n+1/2}+\left|\rho_{j}^{n+1/2}-R_{\mathrm{on},j}^{n+1/2}\right|\right)\right|\\
&&+\Delta t \mathbf{1}_{\mathrm{off},j}\left\|q_{\mathrm{off}}\right\|_{\L{\infty}([0,T])}\left|\rho_{j}^{n+1/2}\right|\\
&\leq&\Delta t\mathbf{1}_{\mathrm{on},j}\left\|q_{\mathrm{on}}\right\|_{\L{\infty}([0,T])}\left|1-\frac{1}{2}\left(\rho_{j}^{n+1/2}+\cancel{R_{\mathrm{on},j}^{n+1/2}}-\cancel{\left|R_{\mathrm{on},j}^{n+1/2}\right|}+\left|\rho_{j}^{n+1/2}\right|\right)\right|\\
&&+\Delta t \mathbf{1}_{\mathrm{off},j}\left\|q_{\mathrm{off}}\right\|_{\L{\infty}([0,T])}\left|\rho_{j}^{n+1/2}\right|\\
&\leq&\Delta t\mathbf{1}_{\mathrm{on},j}\left\|q_{\mathrm{on}}\right\|_{\L{\infty}([0,T])}\left(1+\left|\rho_{j}^{n+1/2}\right|\right)+\Delta t \mathbf{1}_{\mathrm{off},j}\left\|q_{\mathrm{off}}\right\|_{\L{\infty}([0,T])}\left|\rho_{j}^{n+1/2}\right|\\
&\leq&\Delta t\|q_{\mathrm{on}}\|_{\L{\infty}([0,T])}\left(\mathbf{1}_{\mathrm{on},j}+\mathbf{1}_{\mathrm{on},j}\left|\rho_{j}^{n+1/2}\right|\right)+\Delta t\left\|q_{\mathrm{off}}\right\|_{\L{\infty}([0,T])}\mathbf{1}_{\mathrm{off},j}\left|\rho_{j}^{n+1/2}\right|,
\end{eqnarray*}
this case reduces to \eqref{dif_rhon+1_rhon1/2}.\\
The rest of the proof is analogous to Proposition \ref{prop:BV_SpaceTime}.
\end{proof}

\section{Numerical experiments}\label{Sec:Num_Exp}
In this section we present some numerical examples to describe the effects that the ramps have on a road. We solve Model 1 and Model 2 by means \textbf{Algorithm \ref{Algorithm_Scheme}} with the terms $S_{\mathrm{on}}$  \eqref{Son_def1} and \eqref{Son_def2}, respectively. In all numerical examples below, we consider one on-ramp and one off-ramp, both ramps with length $L=0.1$, the on-ramp is located from $x=1.0$ until $x=1.1$,  the off-ramp is located from $x=3$ until $x=3.1$ and we  consider the following kernel functions 
\begin{eqnarray*}
\omega_{\eta}(x)&:=&2\frac{\eta-x}{\eta^2},\\
\omega_{\eta,\delta}(x)&:=& \frac{1}{\eta^6}\frac{16}{5\pi}\left(\eta^2-(x-\delta)^2 \right)^{5/2},
\end{eqnarray*}
for convective and reactive term respectively, with $\eta\in[0,1]$ and $\delta\in[-\eta,\eta]$.

\subsection{Example 1:} Dynamic of Model 1 vs. Model 2.\\
In this example we show numerically the behavior of the density of vehicles in a main road with the presence of one on-ramp and one off-ramp. We solve \eqref{nonRTM} numerically in the interval $[-1,9]$ in simulated times $T=0.5,\ T=2,\ T=5,\ T=7 $. We consider $\Delta x=1/1000$, $\eta=0.05,\ \delta=-0.01$, a constant initial condition $\rho_{0}(x)=0.3$, and the rate of the on- and off-ramp are given by $q_{\mathrm{on}}(t)=1.2,\ q_{\mathrm{off}}(t)=0.8$, respectively.\\
In Fig.\ref{fig:Son13_vs_Son14} we can see that when vehicles enter the ramp, the density of vehicles on the main road increases and a shock wave with negative speed is formed, after that, a rarefaction wave appears and when some vehicles leave the main road through off-ramp a shock wave with positive speed is formed. In particular we can observe a difference between the maximum density that is reached in each model, which may be due to the presence of the term $1-\rho$ in the Model 1.
\begin{figure}[htbp]
     \hspace*{-1pt}
  \begin{tabular}{cc}
    (a)&(b)\\
\includegraphics[scale=0.45]{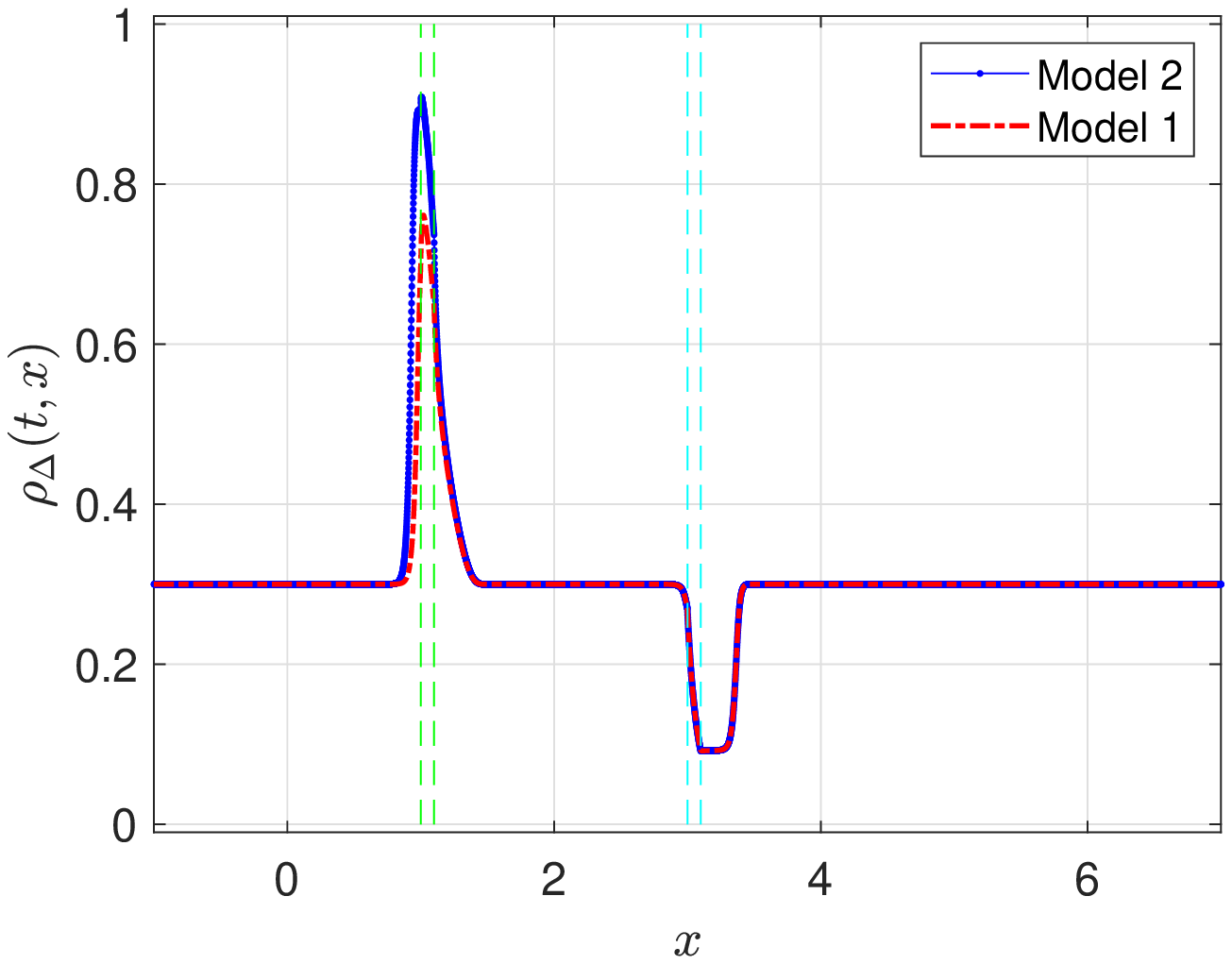}&
\includegraphics[scale=0.45]{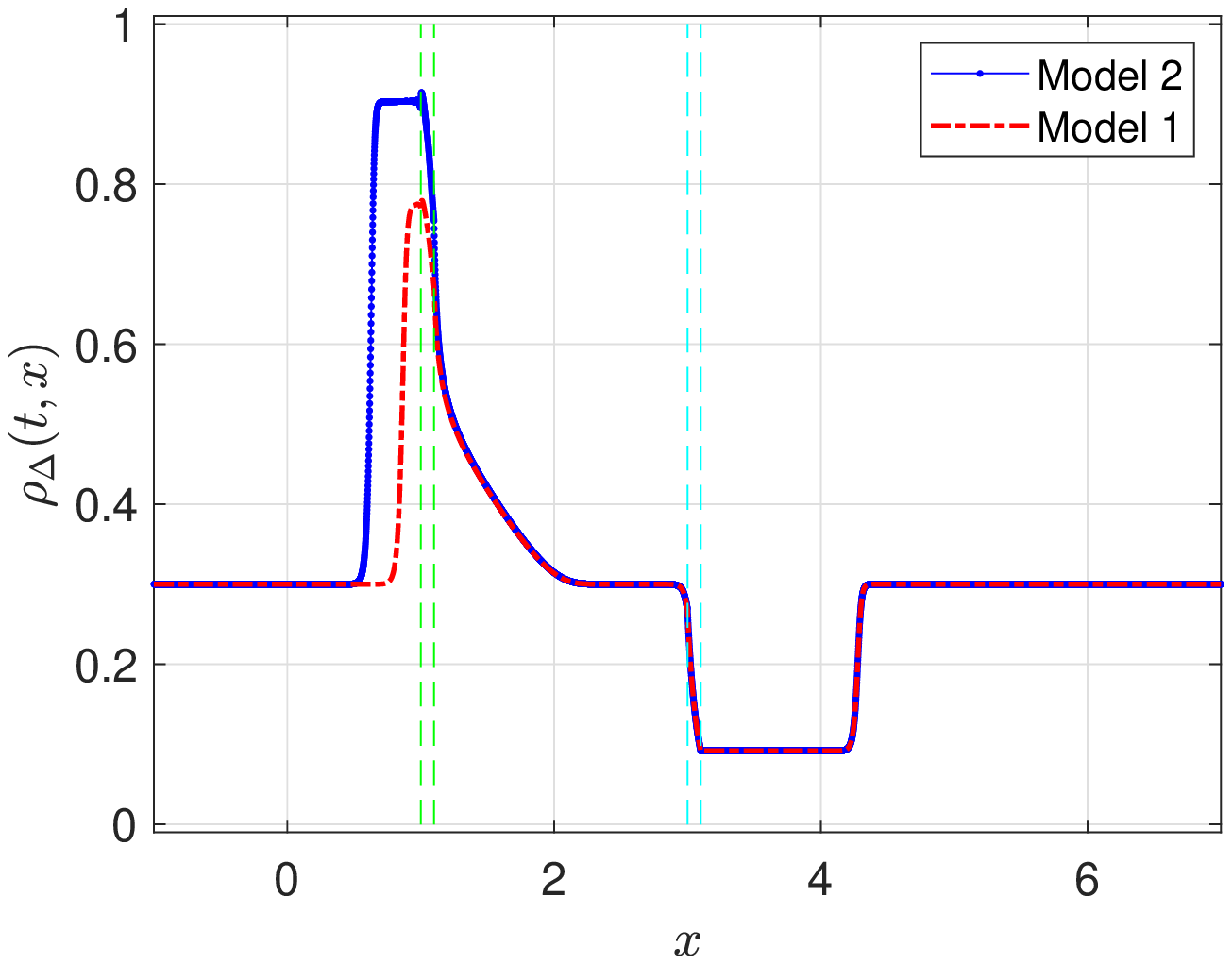}\\
  (c)&(d)\\
\includegraphics[scale=0.45]{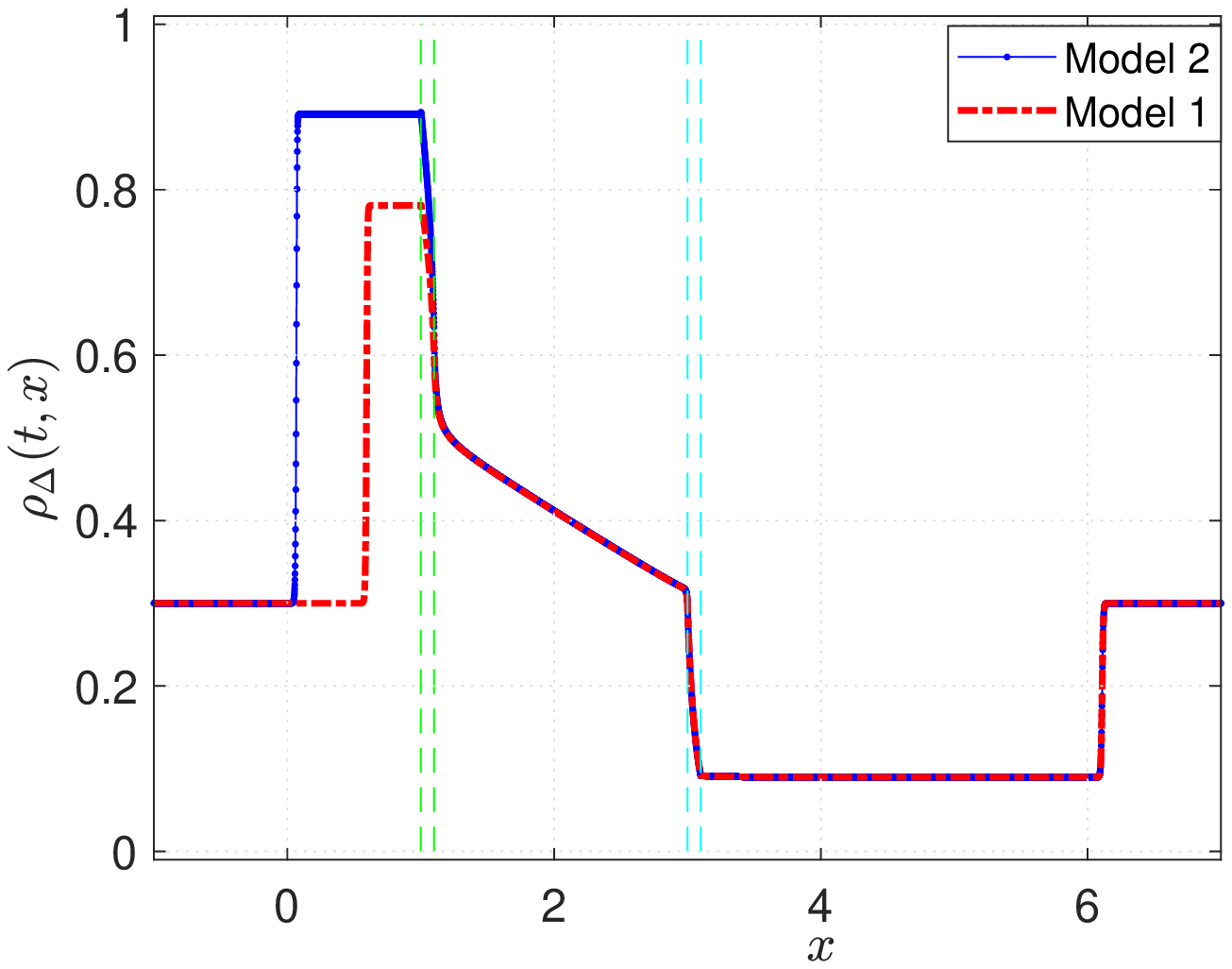}&
\includegraphics[scale=0.45]{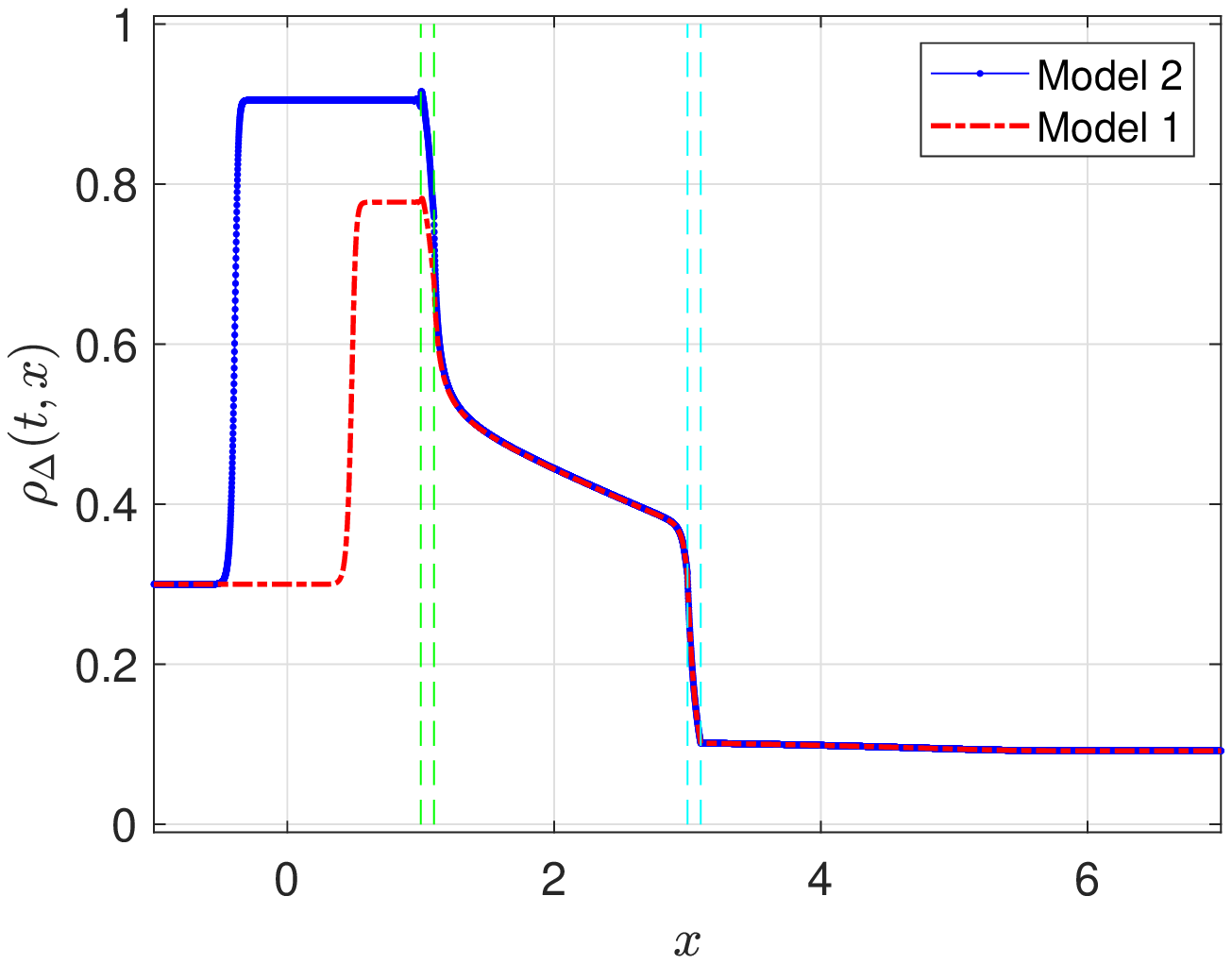}
\end{tabular}
     \caption{Example 1. Numerical approximations of the problem \eqref{nonRTM}. Dynamic of  Model 1 vs. Model 2  at (a)$T=0.5$, (b)$T=2$, (c)$T=5$, (d)$T=7$.}
     \label{fig:Son13_vs_Son14}
 \end{figure}
 \subsection{Example 2:} limit $\eta\to0$ in Model 2.\\
 In this example we take a look at the limit case $\eta\to0$ and investigate the convergence of the Model 2 to the solution of the local problem \eqref{RTM}-\eqref{soff_local}. In particular, we consider the initial condition $\rho_{0}(x)=0.3$ for $x\in[0,1]$, $q_{\mathrm{on}}(t)=1.2, \ q_{\mathrm{off}}(t)=0.8$ at $T=5$ with fixed $\Delta x =1/1000$ and $\eta\in\{0.1,0.05,0.01,0.004\}$, and $\delta=0$. To evaluate the convergence, we compute the $\L{1}$ distance between the approximate solution obtained for the proposed upwind-type scheme by means \textbf{Algorithm \ref{Algorithm_Scheme}} with a given $\eta$ and the result of a classical Godunov scheme for the corresponding local problem. In Table \ref{tab:nonlocal_to_local}, we can observe that the $\L{1}$ distance goes to zero when $\eta\to0.$ The results are illustrated in Fig.\ref{fig:Local_Vs_Nonlocal_Version}.
 \begin{table}[htbp]
     \centering
     \begin{tabular}{|c|c|c|c|c|}
     \hline
     $\eta$ & 0.1 &0.05&0.01&0.004\\ \hline
         $\L{1}$ distance & 2.8e-1&1.6e-1&3.6e-2&1.1e-2 \\ \hline
     \end{tabular}
     \caption{Example 2. $\L{1}$ distance between the approximate solutions to the non-local problem and the local problem for different values of $\eta$ at $T=5$ with $\Delta x=1/1000$.}
     \label{tab:nonlocal_to_local}
 \end{table}
 
 \begin{figure}[htbp]
     \centering
     \includegraphics[scale=0.5]{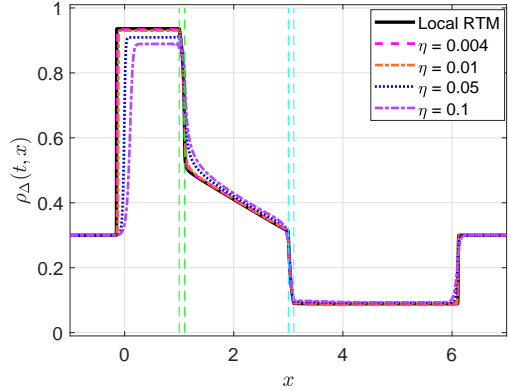}
     \caption{Example 2. Numerical approximations of the problem \eqref{nonRTM} at $T=5$. Comparison of local and non-local versions of the model \eqref{nonRTM} with $\delta=0$ and different values for $\eta$.}
     \label{fig:Local_Vs_Nonlocal_Version}
 \end{figure}
  \subsection{Example 3:} Maximum principle.\\
In this example we verify that the \textbf{Algorithm \ref{Algorithm_Scheme}} with the terms $S_{\mathrm{on}}$ \eqref{Son_def1} and \eqref{Son_def2} satisfy the maximum principle, i.e., we verify numerically that Lemmas \ref{Lemma:Maximum_pple} and \ref{Max_ppio_Son2} respectively, are fulfilled. On the other hand, we also verify that the \textbf{Algorithm \ref{Algorithm_Scheme}} with a discretization of the term $S_{\mathrm{on}}$ \eqref{nolreact1}, which we called Model 0, does not satisfy a maximum principle. For this purpose we consider the initial condition given by 
\begin{equation*}
\rho_{0}(x) = \left\{ \begin{array}{lcc}
             0.1 &   \text{if}  & x \leq 1.1 \\
             0.9 &  \text{if}  & x > 1.1,
             \end{array}
   \right.  
   \end{equation*}
$q_{\mathrm{on}}(t)=1,\ q_{\mathrm{off}}(t)=0.2$ at $T=0.3$, with $\Delta x=1/100$, $\eta=0.05,$ and $\delta=-0.01$.
We can see in Fig.\ref{fig:Maximum_pple} (a) that the Model 0 does not satisfy a maximum principle unlike Model 1 and  Model 2. The Fig\ref{fig:Maximum_pple} (b) is a zoom of (a) in which we can appreciate in a better form that Model 0 does not satisfy a maximum principle. 
 \begin{figure}[htbp]
     \hspace*{-1pt}
  \begin{tabular}{cc}
    (a)&(b)\\
\includegraphics[scale=0.45]{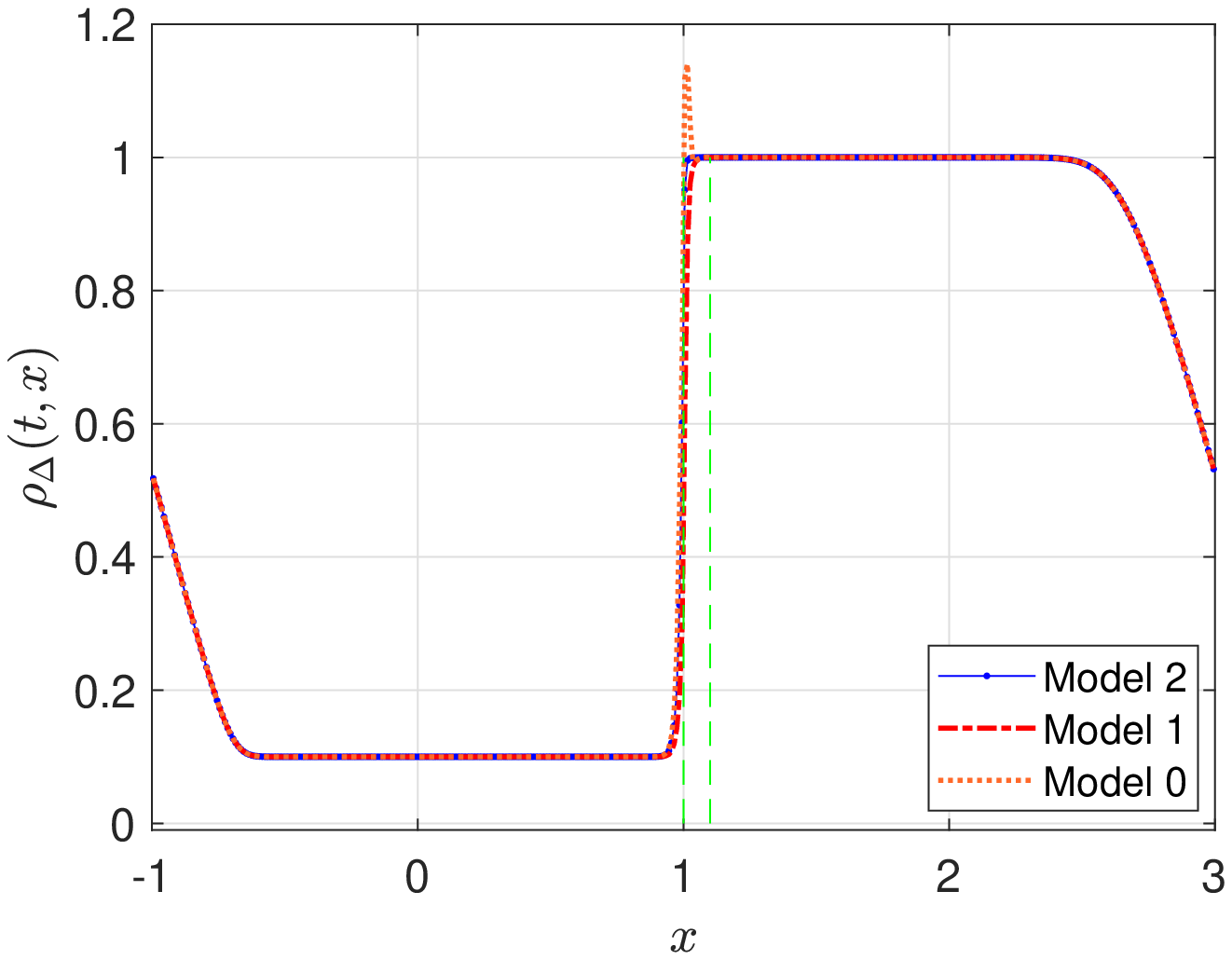}&
\includegraphics[scale=0.45]{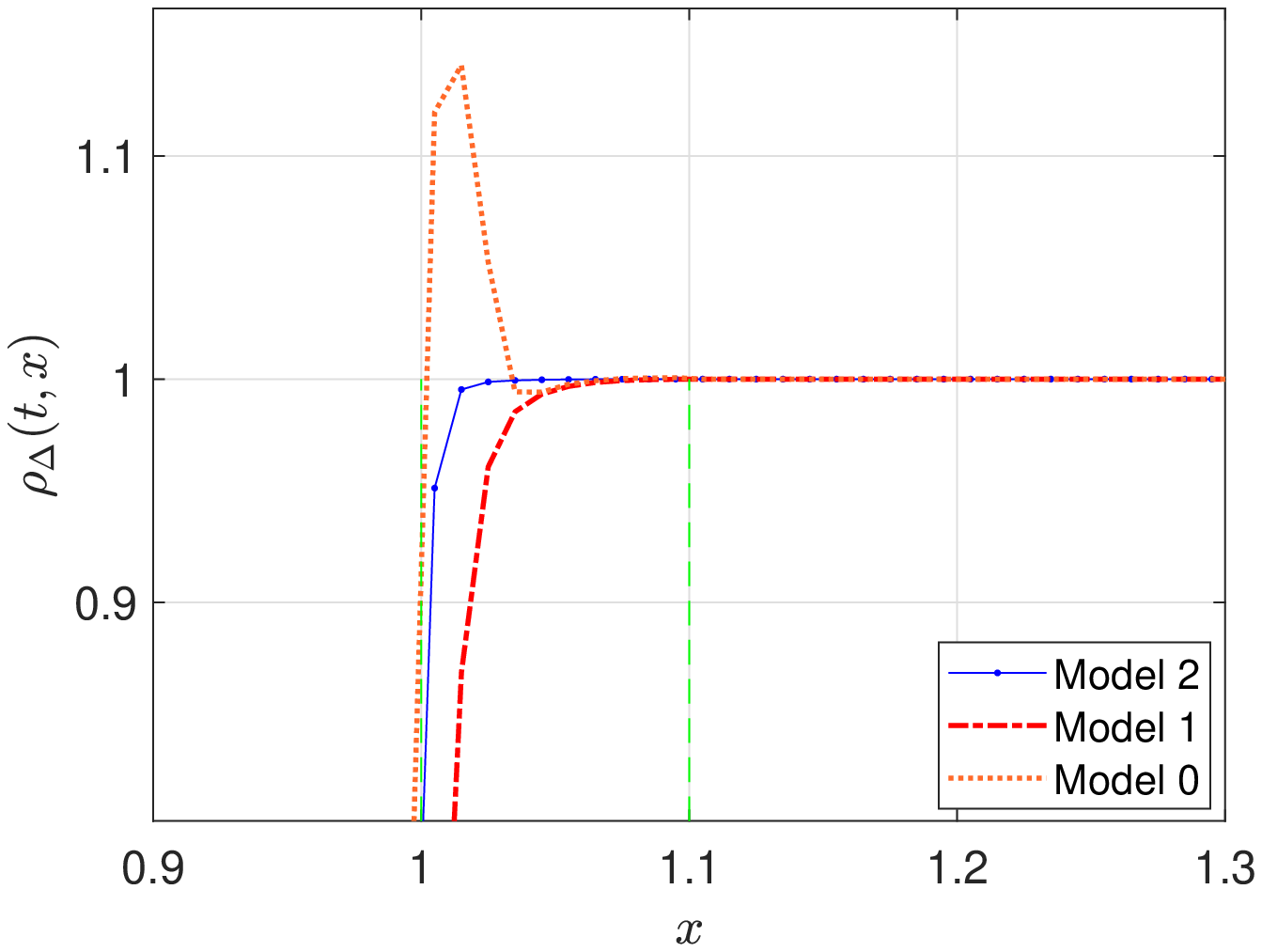}
 \end{tabular}
     \caption{Example 3. Numerical approximation at time $T=0.3$. (a) Model 1, Model 2 satisfying a maximum principle and Model 0 not satisfying a maximum principle. (b) Zoom of a part of (a).}
     \label{fig:Maximum_pple}
 \end{figure}
 \subsection{Example 4:} Free main road.\\
In this example we consider a free main road, i.e, we consider a initial condition $\rho_{0}=0$,  boundary conditions $\rho_{0}(t)=0.4$ for all $t>0$ and absorbing conditions at $x=5$. We also consider the rate of the on-ramp $q_{\mathrm{on}}(t)=\frac{1}{2}\left(\sin(\pi t)+1\right)$ and the rate of the off-ramp $q_{\mathrm{off}}(t)=0.2$.  We solve \eqref{nonRTM} numerically in the interval $[-1,5]$ in different times, namely $T=1,\ T=2,\ T=5,\ T=7 $ and consider $\Delta x=1/1000$, $\eta=0.1,\ \delta=-0.02$. In Fig.\ref{fig:Son13_vs_Son14_Ex4} we can see the dynamic of the model \ref{nonRTM} approximated by means of Model 1 and Model 2. 
 
 \begin{figure}[htbp]
     \hspace*{-1pt}
  \begin{tabular}{cc}
    (a)&(b)\\
\includegraphics[scale=0.45]{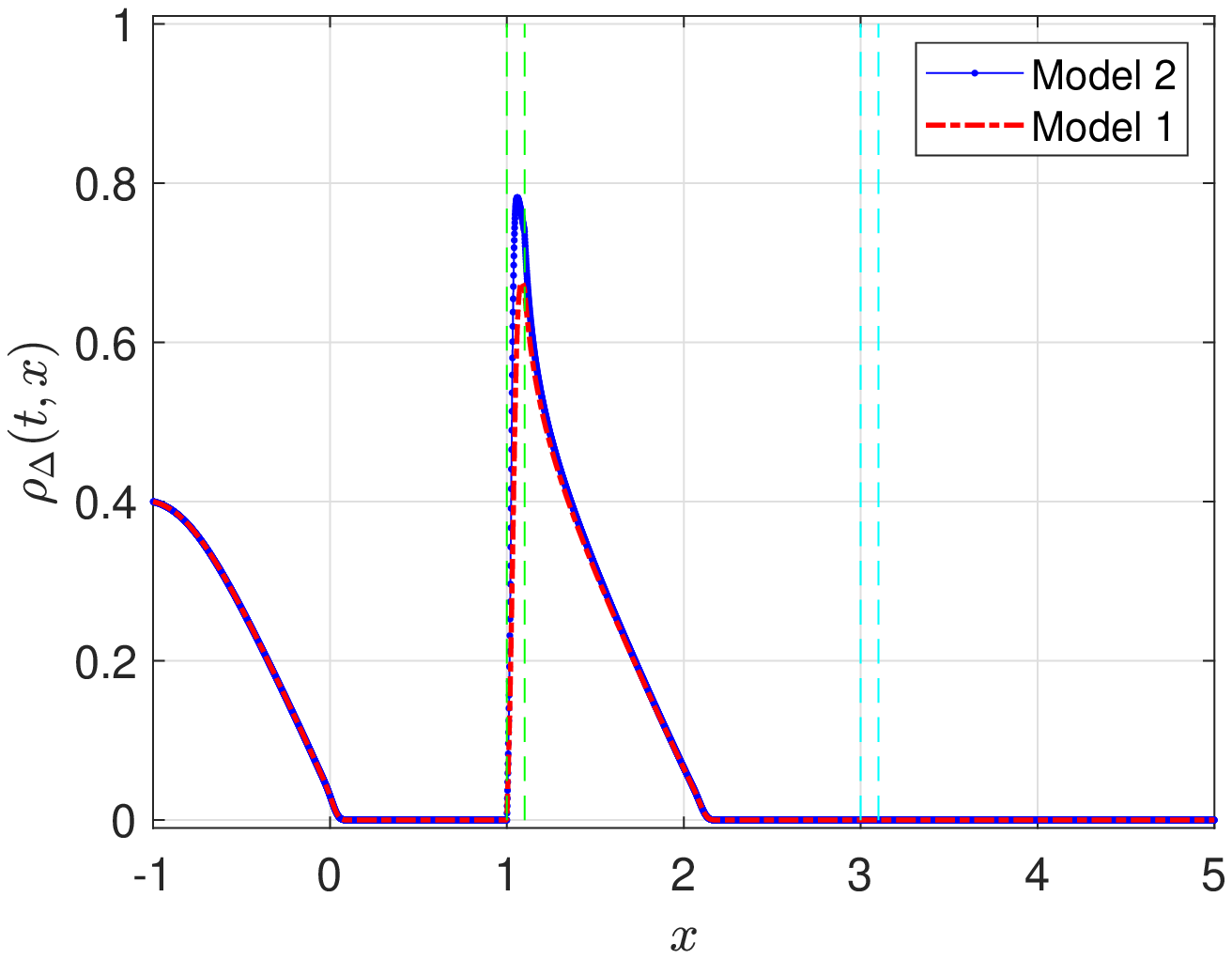}&
\includegraphics[scale=0.45]{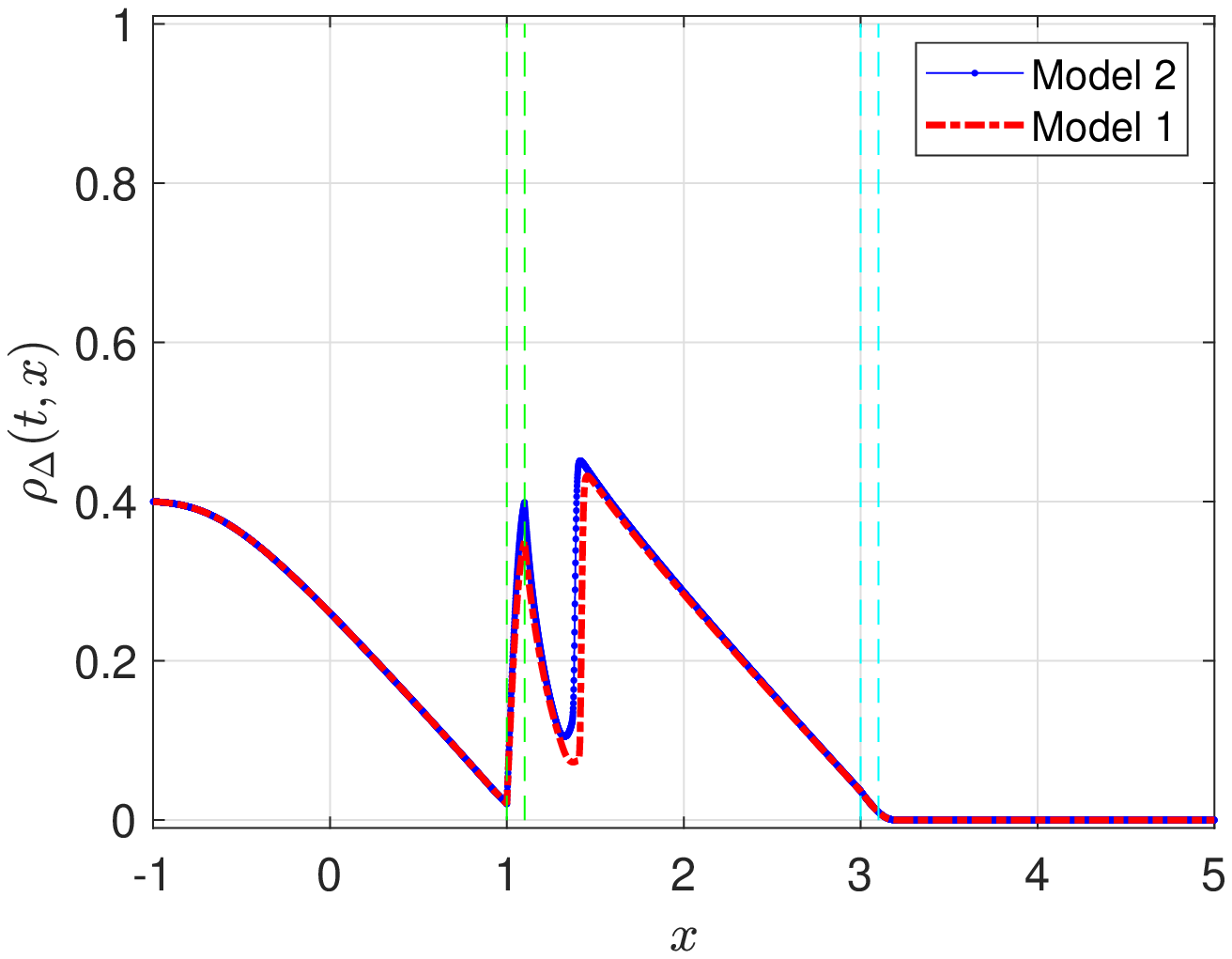}\\
  (c)&(d)\\
\includegraphics[scale=0.45]{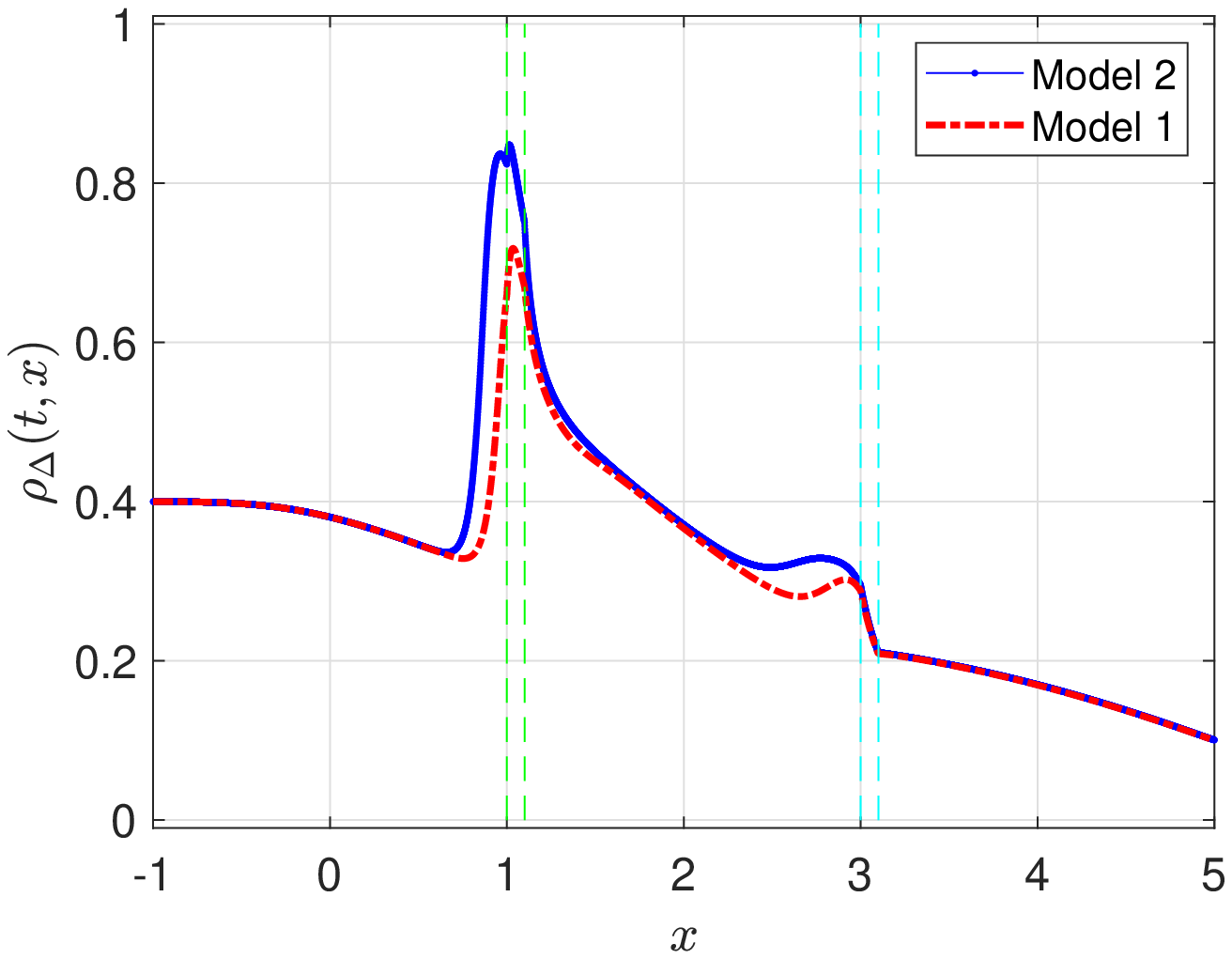}&
\includegraphics[scale=0.45]{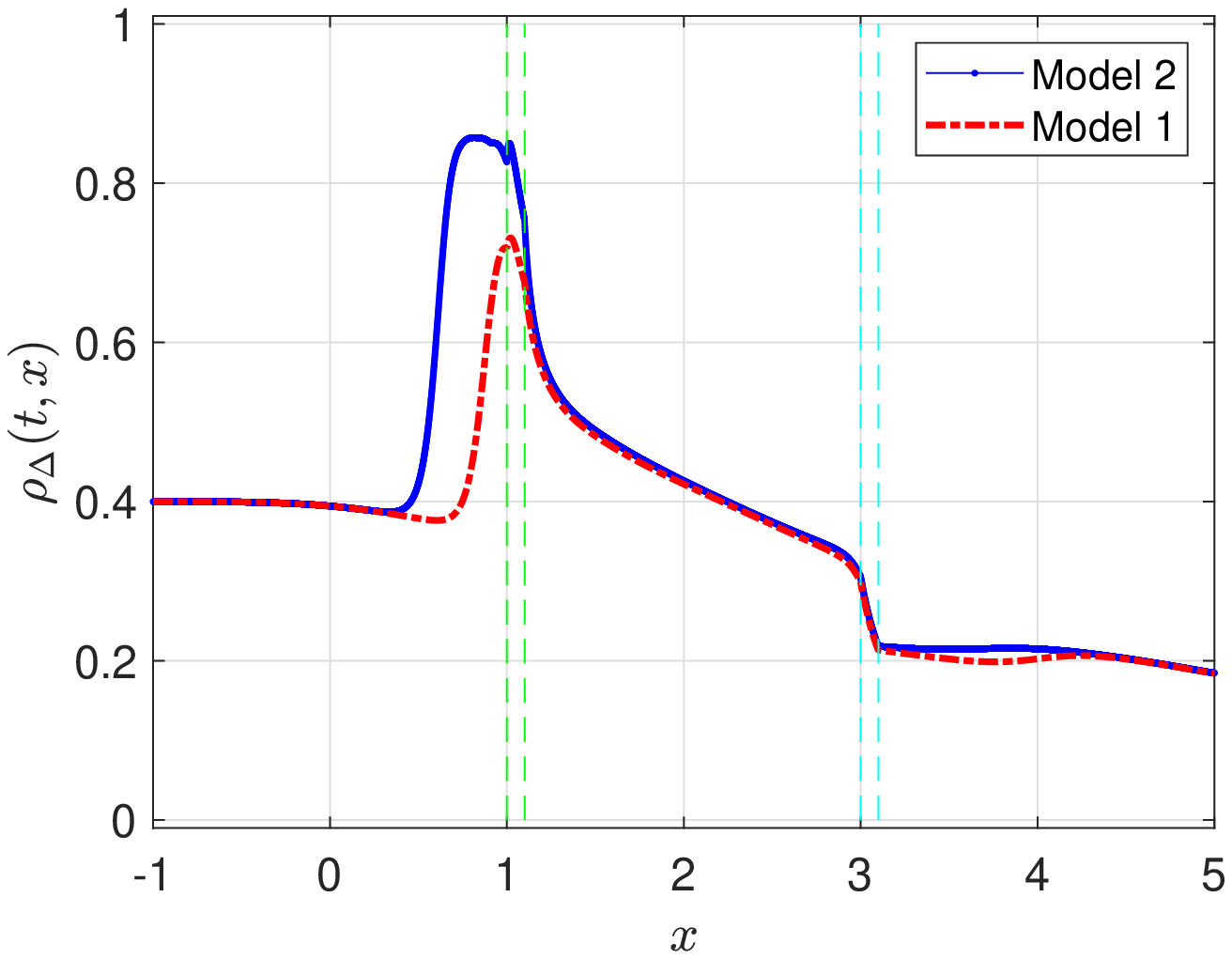}
\end{tabular}
     \caption{Example 4. Dynamic of the model \eqref{nonRTM}. Behavior of the numerical solution computed with \textbf{Algorithm \ref{Algorithm_Scheme}} by means of Model 1 and Model 2 at time (a)$T=1$, (b)$T=2$, (c)$T=5$, (d)$T=7$.}
     \label{fig:Son13_vs_Son14_Ex4}
 \end{figure}

\section{Conclusion and perspectives}
In this paper we introduced a nonlocal balance law to model vehicular traffic flow including on- and off-ramps. We presented three different models called Model 0, Model 1 and Model 2 and we proved existence and uniqueness of solutions for Model 1 and Model 2. We approximated the problem through a upwind-type numerical scheme, providing a Maximum principle, $\L{1}$ and $\BV$ estimates for approximate solutions. Numerical simulations illustrate the dynamics of the studied models and show that Model 0 does not satisfy a maximum principle. A limit model as the kernel support tends to zero is numerically investigated. In a future work, we would like to consider a nonlocal version of second order model proposed in \cite{tie2009new}.

\section*{Acknowledgments}
FAC acknowledges support from “Compagnia di San Paolo” (Torino, Italy).  LMV acknowledges partial support from ANID-Chile through Fondecyt project 1181511 and project AFB170001 of the PIA Program: Concurso Apoyo a Centros Cient\'ificos y Tecnol\'ogicos de Excelencia con Financiamiento Basal.
HDC and LMV  are supported by the INRIA Associated Team ``Efficient numerical schemes for non-local transport phenomena'' (NOLOCO; 2018--2020). HDC was partially supported by the National Agency for Research and Development, ANID-Chile through Scholarship Program, Doctorado Becas Chile 2021, 21210826.

{
\bibliographystyle{siam}
\bibliography{bibliography}}
\end{document}

%% file: carplotnonloc.tex
\begin{tikzpicture}[scale=0.80]
\pgfmathsetmacro{\celllength}{1}
\pgfmathsetmacro{\cellheight}{1}
\pgfmathsetmacro{\sigmah}{0.8}
\pgfmathsetmacro{\muh}{1}
\def\plotfunctwo(#1,#2,#3,#4){{1.5}}
	\draw[domain=-7.5*\celllength:7*\celllength,smooth,variable=\x] 
			plot ({\x},{\plotfunctwo(\x, \cellheight,\sigmah,\muh)});
\fill [red!8, domain=-3.7:-0.7, variable=\x]
      (-3.7, 0)
     -- plot ({\x},{\plotfunctwo(\x, \cellheight,\sigmah,\muh)})
      -- (-0.7, 0)
      -- cycle;

\def\plotfunctwo(#1,#2,#3,#4){{1.5}}
	\draw[domain=-7.5*\celllength:7*\celllength,smooth,variable=\x] 
			plot ({\x},{\plotfunctwo(\x, \cellheight,\sigmah,\muh)});
\fill [gray!20, domain=-7.5:-3.7, variable=\x]
      (-7.5, 0)
     -- plot ({\x},{\plotfunctwo(\x, \cellheight,\sigmah,\muh)})
      -- (-3.7, 0)
      -- cycle;	

\def\plotfunctwo(#1,#2,#3,#4){{1.5}}
	\draw[domain=-7.5*\celllength:7*\celllength,smooth,variable=\x] 
			plot ({\x},{\plotfunctwo(\x, \cellheight,\sigmah,\muh)});
\fill [gray!20, domain=-.7:7, variable=\x]
      (-.7, 0)
     -- plot ({\x},{\plotfunctwo(\x, \cellheight,\sigmah,\muh)})
      -- (7,0)
      -- cycle;

\def\plotfunctwo(#1,#2,#3,#4){{0}}
	\draw[domain=-7.5*\celllength:7*\celllength,smooth,variable=\x] 
			plot ({\x},{\plotfunctwo(\x, \cellheight,\sigmah,\muh)});
\fill [gray!20, domain=-3:-1, variable=\x]
      (-7.5, -2)
     -- plot ({\x},{\plotfunctwo(\x, \cellheight,\sigmah,\muh)})
      -- (-5,-2)
      -- cycle;

\def\plotfunctwo(#1,#2,#3,#4){{0}}
	\draw[domain=-7.5*\celllength:7*\celllength,smooth,variable=\x] 
			plot ({\x},{\plotfunctwo(\x, \cellheight,\sigmah,\muh)});
\fill [gray!20, domain=1:3, variable=\x]
      (3, -2)
     -- plot ({\x},{\plotfunctwo(\x, \cellheight,\sigmah,\muh)})
      -- (7,-2)
      -- cycle;	
	\draw [fill=green!80] (-7.5,0) -- (-3,0) -- (-7.5,-2) -- cycle;
     \draw [fill=green!80] (-5,-2) -- (-1,0) -- (1,0) --(5,-2) -- cycle;
     \draw [fill=green!80] (3,0) -- (7,0) --(7,-2) -- cycle; 
\foreach \x in {-7.5}
{
    \shade[top color=white, bottom color=yellow, shading angle={135}]
        [draw=black,fill=red!20,rounded corners=0.7ex,very thick]
        (\x+0.01,0.7) -- ++(0,0.26) --  ++(0.8,0) -- ++(0.2,-0.06) -- ++(0,-0.2) -- (\x+0.01,0.7) -- cycle;
    \draw[very thick, rounded corners=0.5ex,fill=cyan!50,thick]  (\x+0.17,0.96) -- ++(0.12,0.14) -- ++(0.32,0) -- ++(0.2,-0.14) -- (\x+0.17,0.9	6);    
    \draw[draw=black,fill=black!50,thick] (\x+0.26,0.7) circle (.1);
    \draw[draw=black,fill=black!50,thick] (\x+0.81,0.7) circle (.1);
}
\foreach \x in {-5.5}
{
	
    \shade[top color=white!50, bottom color=orange, shading angle={135}]
        [draw=black,fill=red!20,rounded corners=0.7ex,very thick]
        (\x+0.01,0.7) -- ++(0,0.26) --  ++(0.8,0) -- ++(0.2,-0.06) -- ++(0,-0.2) -- (\x+0.01,0.7) -- cycle;
    \draw[very thick, rounded corners=0.5ex,fill=cyan!50,thick]  (\x+0.17,0.96) -- ++(0.12,0.14) -- ++(0.32,0) -- ++(0.2,-0.14) -- (\x+0.17,0.9	6);    
    \draw[draw=black,fill=black!50,thick] (\x+0.26,0.7) circle (.1);
    \draw[draw=black,fill=black!50,thick] (\x+0.81,0.7) circle (.1);
}

\foreach \x in {-4, -2.5, -1.4}
{
	
    \shade[top color=white!50, bottom color=red, shading angle={135}]
        [draw=black,fill=red!20,rounded corners=0.7ex,very thick]
       (\x+0.01,0.7) -- ++(0,0.26) --  ++(0.8,0) -- ++(0.2,-0.06) -- ++(0,-0.2) -- (\x+0.01,0.7) -- cycle;
    \draw[very thick, rounded corners=0.5ex,fill=cyan!50,thick]  (\x+0.17,0.96) -- ++(0.12,0.14) -- ++(0.32,0) -- ++(0.2,-0.14) -- (\x+0.17,0.9	6);    
    \draw[draw=black,fill=black!50,thick] (\x+0.26,0.7) circle (.1);
    \draw[draw=black,fill=black!50,thick] (\x+0.81,0.7) circle (.1);
}

\foreach \x in {0.}
{
	
Car has a length of 1
Car begin
Body
    \shade[top color=white!50, bottom color=orange, shading angle={135}]
        [draw=black,fill=red!20,rounded corners=0.7ex,very thick]
        (\x+0.01,0.7) -- ++(0,0.26) --  ++(0.8,0) -- ++(0.2,-0.06) -- ++(0,-0.2) -- (\x+0.01,0.7) -- cycle;
Windows        
    \draw[very thick, rounded corners=0.5ex,fill=cyan!50,thick]  (\x+0.17,0.96) -- ++(0.12,0.14) -- ++(0.32,0) -- ++(0.2,-0.14) -- (\x+0.17,0.9	6);    
Wheels
    \draw[draw=black,fill=black!50,thick] (\x+0.26,0.7) circle (.1);
    \draw[draw=black,fill=black!50,thick] (\x+0.81,0.7) circle (.1);
Car end
}

\foreach \x in {2.5}
{
    \shade[top color=white, bottom color=yellow, shading angle={135}]
        [draw=black,fill=red!30,rounded corners=0.7ex,very thick]
        (\x+0.01,0.7) -- ++(0,0.26) --  ++(0.8,0) -- ++(0.2,-0.06) -- ++(0,-0.2) -- (\x+0.01,0.7) -- cycle;
    \draw[very thick, rounded corners=0.5ex,fill=cyan!50,thick]  (\x+0.17,0.96) -- ++(0.12,0.14) -- ++(0.32,0) -- ++(0.2,-0.14) -- (\x+0.17,0.9	6);    
    \draw[draw=black,fill=black!50,thick] (\x+0.26,0.7) circle (.1);
    \draw[draw=black,fill=black!50,thick] (\x+0.81,0.7) circle (.1);
}

\foreach \x in {-4.5}
{
\begin{scope}[rotate=30,shift={(	1.6,2.)}]
    \shade[top color=white, bottom color=orange, shading angle={135}]
        [draw=black,fill=red!20,rounded corners=0.7ex,very thick,]
        (\x+0.01,-1) -- ++(0,0.26) --  ++(0.8,0) -- ++(0.2,-0.058) -- ++(0,-0.2) -- (\x+0.01,-1) -- cycle;
       
    \draw[very thick, rounded corners=0.5ex,fill=cyan!50,thick]  (\x+0.17,-0.74) -- ++(0.12,0.14) -- ++(0.32,0) -- ++(0.2,-0.14) -- (\x+0.17,-0.74);    
    \draw[draw=black,fill=black!50,thick] (\x+0.26,-1) circle (.1);
    \draw[draw=black,fill=black!50,thick] (\x+0.81,-1) circle (.1);
\end{scope}
}

\foreach \x in {2.2}
{
\begin{scope}[rotate=-30,shift={(1,-0.9)}]
    \shade[top color=white, bottom color=blue, shading angle={135}]
        [draw=black,fill=red!20,rounded corners=0.7ex,very thick,]
        (\x+0.01,2) -- ++(0,0.26) --  ++(0.8,0) -- ++(0.2,-0.06) -- ++(0,-0.2) -- (\x+0.01,2) -- cycle;
    \draw[very thick, rounded corners=0.5ex,fill=cyan!50,thick]  (\x+0.17,2.26) -- ++(0.12,0.14) -- ++(0.32,0) -- ++(0.2,-0.14) -- (\x+0.17,2.26);    
    \draw[draw=black,fill=black!50,thick] (\x+0.26,2) circle (.1);
    \draw[draw=black,fill=black!50,thick] (\x+0.81,2) circle (.1);
\end{scope}
}

\draw[-,semithick] (-7.5,1.5) -- (7,1.5);   
 \draw[-,semithick] (-7.5,0) -- (-3,0);
\draw[dotted](-3,0)--(-1,0); 
\draw[-,semithick] (-1,0) -- (1,0);	
\draw[dotted] (1,0) -- (3,0);	
\draw[-,semithick] (3,0) -- (5,0);	
\draw[-,semithick](-7.5,-2)--(-3,0);
\draw[-,semithick](-1,0)--(-5,-2);
\draw[dotted] (-3.7,0) -- (-3.7,1.5);	
\draw[dotted] (-0.7,0) -- (-0.7	,1.5);
\draw (-6,-2.3) node {$q_{\mathrm{on}}$};
\draw (-3.7,1.7) node {\footnotesize{$x-\eta+\delta$}};
\draw (-0.7,1.7) node {\footnotesize{$x+\eta+\delta$}};
\draw (-2.2,1.7) node {\footnotesize{$x$}};
\draw[-,semithick](1,0)--(5,-2);
\draw[-,semithick](3,0)--(7,-2);
\draw (6,-2.3) node {$q_{\mathrm{off}}$};

%
\end{tikzpicture}